\documentclass[a4paper]{amsart}
\usepackage{xypic,xcolor}
\usepackage{tikz,stackrel}
\usepackage{hyperref,amssymb,bm}
\usepackage[all]{xy}

\def\sw#1{{\sb{(#1)}}}

\def\id{{\rm id}}

\newtheorem{thm}{Theorem}[section]

\newtheorem{lem}[thm]{Lemma}
\newtheorem{cor}[thm]{Corollary}
\theoremstyle{definition}
\newtheorem{dfn}[thm]{Definition}

\theoremstyle{remark}
\newtheorem{rmk}[thm]{Remark}

\numberwithin{equation}{section}

\newcommand{\CC}{\mathbb{C}}
\newcommand{\NN}{\mathbb{N}}
\newcommand{\RR}{\mathbb{R}}
\newcommand{\TT}{\mathbb{T}}
\newcommand{\ZZ}{\mathbb{Z}}

\newcommand{\Ee}{\mathcal{E}}
\newcommand{\Kk}{\mathcal{K}}

\newcommand{\Tt}{\mathcal{T}}

\newcommand{\PTt}{\mathbb{P}^N(\Tt)}

\begin{document}
\parskip=0.75\baselineskip
\parindent=4mm

\author[P.M.~Hajac]{Piotr~M.~Hajac}
\address{Instytut Matematyczny, Polska Akademia Nauk, ul.~\'Sniadeckich 8, Warszawa, 00--656 Poland}
\email{pmh@impan.pl}
\author[R.~Nest]{Ryszard Nest}
\address{Department of Mathematics, Copenhagen University, Universitetsparken 5, 2100 Copenhagen, Denmark }
\email{rnest@math.ku.dk}
\author[D.~Pask]{David Pask}
\address{School of Mathematics and Applied Statistics, University of Wollongong, NSW 2522, Australia}
\email{dpask@uow.edu.au}
\author[A.~Sims]{Aidan Sims}
\address{School of Mathematics and Applied Statistics, University of Wollongong, NSW 2522, Australia}
\email{asims@uow.edu.au}
\author[B.~Zieli\'nski]{Bartosz~Zieli\'nski}
\address{
Department of  Computer Science, University of \L{}\'od\'z, Pomorska 149/153 90-236
\L{}\'od\'z, Poland} \email{bzielinski@uni.lodz.pl}
\title[Quantum odd spheres and complex projective spaces]{\vspace*{-25mm}
The $K$-theory of twisted  multipullback quantum\\ odd spheres and complex projective spaces}

\begin{abstract}\vspace*{-2.5mm}
We find multipullback quantum odd-dimensional spheres equip\-ped
with natural $U(1)$-actions that yield
 the multipullback quantum complex projective spaces constructed from Toeplitz cubes as noncommutative quotients.
We prove  that the noncommutative line bundles associated to multipullback
quantum odd spheres are pairwise stably \emph{non}-isomorphic, and that the \mbox{$K$-groups} of
multipullback quantum complex projective spaces and odd spheres coincide with their
classical counterparts. We show that these $K$-groups remain the same for more general
twisted versions of our quantum odd spheres and complex projective spaces.
\end{abstract}
\maketitle
\vspace*{-7mm} \setcounter{tocdepth}{1}
{\parskip=0.55\baselineskip\footnotesize\tableofcontents}

\vspace*{-20mm}\section*{Introduction}
\noindent
Complex projective space is a fundamental object in topology and algebraic geometry. It
also makes its mark in lattice theory as its affine covering provides a natural model of
a free distributive lattice~\cite{b-g67}. In \cite{hkz12},  a noncommutative deformation
of complex projective spaces preserving this lattice-theoretic property was introduced
and studied. The new quantum complex projective
space \mbox{$C^*$-algebras} $C(\mathbb{P}^N(\Tt))$ were defined as multipullback
$C^*$-algebras~\cite{p-gk99} rather than as fixed-point subalgebras~\cite{vs91,m-u95}.

In this paper, we    solve the  problem of constructing multipullback quantum-odd-sphere
$C^*$-algebras $C(S^{2N+1}_H)$ from which the $C^*$-algebras $C(\mathbb{P}^N(\Tt))$
emerge as fixed-point subalgebras for a natural circle action. Then we develop and utilise a
presentation of $C(S^{2N+1}_H)$ as the
universal $C^*$-algebra generated by $N+1$ commuting isometries satisfying a sphere
equation (see Theorem~\ref{prp:pullback isomorphism}).
We exploit this presentation to show that the $K$-goups of $C(S^{2N+1}_H)$ and of
$C(\mathbb{P}^N(\Tt))$ coincide with their classical counterparts.
\footnote{Keywords: free action on C*-algebras, associated noncommutative line bundle,
multipullback and higher-rank graph \mbox{C*-algebras}, noncommutative deformation. 
AMS codes: 46L80, 46L85.}

The constructions and results described above admit the following generalisation. For
each antisymmetric matrix $\theta \in M_{N+1}(\RR)$, we construct $\theta$-twisted
versions $C(S^{2N+1}_{H, \theta})$ and $C(\mathbb{P}^N_\theta(\Tt))$ of our quantum odd
sphere $C^*$-algebra and our quantum complex projective space $C^*$-algebra. The twisted
sphere algebra is universal for $N+1$ isometries commuting up to phases specified by the
matrix $\theta$ and satisfying a sphere equation. The twisted projective space
$C^*$-algebra is the fixed-point subalgebra of $C(S^{2N+1}_{H,\theta})$ for a natural
diagonal $U(1)$-action. We prove that $K$-theory of these algebras is independent of
$\theta$.

To state our main result, we recall some background. Given a $C^*$-algebra~$A$, we write
$C(U(1), A)$ for the $C^*$-algebra of norm-continuous functions from $U(1)$ to $A$. Each
action $\alpha$ of $U(1)$ on $A$ determines a homomorphism
\begin{equation}\label{deltaalpha}
\delta : A \longrightarrow C(U(1),A)\quad \text{by}\quad \delta(a)(\lambda) := \alpha_\lambda(a),\quad \ a \in A ,\; \lambda \in U(1).
\end{equation}
We say that $\alpha$ is \emph{free} if and only if
\begin{equation*}
\overline{\operatorname{span}}\{a\, \delta(b) \mid a,b \in A\} = C(U(1), A),
\end{equation*}
where $\overline{\operatorname{span}}$ stands for the closed linear span. The general
definition of freeness of a quantum-group action on a $C^*$-algebra is due to
Ellwood~\cite{Ellwood:JFA2000}, and the special case of any compact Hausdorff topological
group acting on a unital $C^*$-algebra looks exactly as above.

Given $\alpha : U(1) \curvearrowright A$ as above, for each character $m\in\widehat{U(1)}\cong\mathbb{Z}$, the
\emph{spectral subspace} $A_m$ is
\begin{equation*}
A_m := \{a \in A \mid \alpha_\lambda(a) = \lambda^m a\text{ for all }\lambda \in U(1)\}.
\end{equation*}
The subspace $A_0$ is the fixed-point subalgebra $A^\alpha$ (also denoted $A^{U(1)}$) of $A$, and since $A_m A_n
\subseteq A_{m+n}$ for all $m,n$, the spectral subspaces are always $A^\alpha$-bimodules.
When $\alpha$ is free, they are finitely generated projective left
$A^\alpha$-modules \cite[Theorem~1.2]{dy13} encoding associated noncommutative line
bundles.

By constructing a strong connection~\cite{h-pm96}, we prove that the action of $U(1)$ on $C(S^{2N+1}_H)$ is free,
so that its spectral subspaces
$C(S^{2N+1}_H)_n$ are finitely generated projective left
$C(\mathbb{P}^N(\Tt))$-modules.
To prove that the characters of $U(1)$ defining these noncommutative line
bundles are $K_0$-invariants, we derive a general method of
pulling back noncommutative associated line bundles over equivariant maps (Theorem~\ref{thm:epullback}).

The key results of this paper can be summarized as follows:
\begin{thm}\label{thm:main}
Fix an integer $N \ge 1$ and a matrix $\theta \in M_{N+1}(\RR)$ that is antisymmetric in
the sense that $\theta_{ij} = -\theta_{ji}$ for all $i,j$. Then:
\begin{enumerate}
    \item\label{it:main-sphere-Kth}
$K_0(C(S^{2N+1}_{H,\theta})) \cong \ZZ \cong K_1(C((S^{2N+1}_{H,\theta}))$.
    \item\label{it:main-PS-Kth} $K_0(C(\mathbb{P}^N_\theta(\Tt)))=\mathbb{Z}^{N+1}$
        and $K_1(C(\mathbb{P}^N_\theta(\Tt)))=0$.
    \item\label{it:main-stabnontriv} The spectral subspaces $C(S^{2N+1}_H)_m$,
        regarded as left $C(\mathbb{P}^N(\Tt))$-modules, are pairwise stably
        nonisomorphic. In particular, the module  $C(S^{2N+1}_H)_{-1}$ of sections of the
        tautological line bundle is not stably free.
\end{enumerate}
\end{thm}

Our multipullback approach to quantum odd spheres is based on the Heegaard-type splitting
of a $(2N+1)$-dimensional sphere into $N$-dimensional solid tori. Each odd-dimensional
sphere decomposes into a union of solid tori, along the lines of the Heegaard splitting
of the 3-sphere~\cite{h-p98}. Under this decomposition, the embedding of each component
torus in the sphere is equivariant for the diagonal $U(1)$-action. 
Taking quotients by the $U(1)$-actions yields a covering of the complex projective space
by quotients of solid tori, which is a closed restriction of the usual affine covering.

To obtain the untwisted ($\theta=0$) sphere algebras $C(S^{2N+1}_H)$, we study a noncommutative deformation of
this decomposition, using the point of view from \cite{KlimekLesniewski:JFA1993} that the
Toeplitz algebra $\mathcal{T}$ can be regarded as the $C^*$-algebra of a noncommutative
disc. In \cite{CM2000}, the authors constructed a decomposition of a 3-dimensional
quantum sphere along these lines by taking a pullback of two copies of the tensor product
of the circle algebra and the Toeplitz algebra. The index pairing of noncommutative line bundles over the resulting
 pullback quantum complex projective line (mirror quantum sphere)
was computed in~\cite{hms06b}.
Subsequently, in his Ph.D.\ thesis, Jan
Rudnik extended the construction in \cite{CM2000} to five dimensions using multipullback $C^*$-algebras.
One of his main results was establishing the stable nontriviality of the dual
tautological line bundle over the multipullback complex quantum projective
plane~\cite[Theorem~2.4]{hr}. In this paper, we carry this idea further to all odd
integers bigger than one. Very recently, Albert Jeu-Liang Sheu showed in \cite{s-ajl} that, for all dimensions, the multipullback
quantum-complex-projective-space $C^*$-algebras can be realized as groupoid $C^*$-algebras.

The paper is organized as follows. In Section~\ref{sec:background}, we recall   definitions and claims
crucial for the formulation and proofs of new results.
In Section~\ref{sec:twisted spheres}, we construct our multipullback quantum-odd-sphere $C^*$-algebras
and their twisted analogues. With the help of the theory of twisted higher-rank
graph $C^*$-algebras~\cite{SimsWhiteheadEtAl:xx13}, we establish that the twisted multipullback
quantum-odd-sphere $C^*$-algebras can be
presented in terms of a universal property (see Theorem~\ref{prp:pullback isomorphism}).
In Section~\ref{sec:projective}, we construct
 quantum complex projective space $C^*$-algebras and their twisted analogues as
fixed-point algebras for $U(1)$-actions on the corresponding sphere algebras. We identify
the untwisted quantum-projective-space algebras obtained in this way with the ones
constructed in \cite{hkz12} as multipullbacks. In Section~\ref{sec:PSpace Kth}, we prove
parts (\ref{it:main-sphere-Kth}) and (\ref{it:main-PS-Kth}) of Theorem~\ref{thm:main}.
In Section~\ref{sec:line bundles}, we use the  Chern-Galois theory of
\cite{bh04} to prove Theorem~\ref{thm:epullback}, which then we use to show
 Theorem~\ref{thm:main}\eqref{it:main-stabnontriv}.

\section{Background}\label{sec:background}

\subsection{Multipushouts, multipullbacks and the cocycle condition}
In what follows, we will construct algebras of functions on quantum spaces as
multipullbacks of $C^*$-algebras. To make sure that this construction corresponds via
duality to the presentation of a quantum space as a union of closed subspaces (see
\cite{HZ}), we assume the  cocycle condition (see Definition~\ref{def:cocycle}). First we
need some auxiliary definitions.

Let $(\pi^i_j:A_i\rightarrow A_{ij})_{i,j\in J,i\neq j}$ be a finite family of surjective
$C^*$-algebra homomorphisms, with $A_{ij} = A_{ji}$ for $i \neq j$. For all distinct
$i,j,k\in J$, we define $A^i_{jk}:=A_i/(\ker\pi^i_j+\ker\pi^i_k)$ and denote by
\mbox{$[\cdot]^i_{jk}:A_i\rightarrow A^i_{jk}$} the canonical surjections. For distinct
$i,j,k \in J$, define
\begin{equation*}
\pi^{ij}_k:A^i_{jk}\longrightarrow A_{ij}/\pi^i_j(\ker\pi^i_k),\quad\text{ by }\quad
[b_i]^i_{jk}\longmapsto\pi^i_j(b_i)+\pi^i_j(\ker\pi^i_k).
\end{equation*}
These $\pi^{ij}_k$ are isomorphisms when the $\pi^i_j$ are all surjective, as assumed
herein.

\begin{dfn}[Proposition~9 in \cite{CM2000}] \label{def:cocycle}
We say that a finite family $(\pi^i_j:A_i\rightarrow A_{ij})_{i,j\in J,i\neq j}$ of
surjective $C^*$-homomorphisms satisfies the {\em cocycle condition} if and only if, for
all distinct $i,j,k\in J$,
\begin{enumerate}
\item $\pi^i_j(\ker\pi^i_k)=\pi^j_i(\ker\pi^j_k)$, and
\item the isomorphisms
    $\phi^{ij}_k:=(\pi^{ij}_k)^{-1}\circ\pi^{ji}_k:A^j_{ik}\rightarrow A^i_{jk}$
    satisfy $\phi^{ik}_j=\phi^{ij}_k\circ\phi^{jk}_i$.
\end{enumerate}
\end{dfn}

Theorem~1 of \cite{HZ} implies that a finite family $(\pi^i_j:A_i\rightarrow
A_{ij})_{i,j\in J,i\neq j}$ of $C^*$-algebra surjections satisfies the cocycle condition
if and only if, for all $K\subsetneq J$, all $k\in J\setminus K$, and all $(b_i)_{i\in
K}\in \bigoplus_{i\in K}A_i$ such that $\pi^i_j(b_i)=\pi^j_i(b_j)$ for all distinct
$i,j\in K$, there exists $b_k\in A_k$ such that also $\pi^i_k(b_i)=\pi^k_i(b_k)$ for all
$i\in K$. This corresponds in the classical setting to the idea that all partial pushouts
of a collection of topological spaces embed in the total pushout.

\subsection{Heegaard-type splittings of odd spheres}\label{sec:Hsplittings}

We recall the Heegaard-type splittings of odd-dimensional spheres into solid tori. We
write $$\mathbb{T} := \{c \in \CC \;|\; |c| = 1\}$$ for the unit circle,  $D := \{c \in \CC
\;|\; |c| \le 1\}$ for the unit disc, and $$S^{2N+1} := \{(z_i)_i \in \CC^{N+1} \mid
\sum_{i=0}^N |z_i|^2 := 1\}$$ for the unit $(2N+1)$-dimensional sphere. For $0 \le i \le
N$, let
\begin{equation*}
V_i := \big\{(z_0,\ldots,z_N) \in S^{2N+1} \mid |z_i| = \max\{|z_0|, \dots, |z_N|\}\big\}.
\end{equation*}
Also, let $z:=(z_0,\ldots,z_N)$ and $d:=(d_0,\ldots,d_N)$. Then $\phi_i(z) :=
|z_i|^{-1}z$ determines a homeomorphism $\phi_i : V_i \to D^i \times \TT \times D^{N-i}
\subseteq \CC^{N+1}$, with inverse given by $\phi^{-1}_i(d) = (1+\sum_{j\neq
i}|d_j|^2)^{-\frac{1}{2}}\,d$.

These homeomorphisms allow us to present $S^{2N+1}$  as a multipushout of closed solid tori. Indeed, for each~$i$, let
$X_i := D^{i}\times \mathbb{T}\times D^{N-i}$, and for $i < j$, let
\begin{equation*}
X_{i,j} := D^{i}\times \mathbb{T}\times D^{j-i-1} \times \mathbb{T} \times D^{N-j} = X_i \cap X_j.
\end{equation*}
Then $S^{2N+1}$ is the multipushout of the solid tori $X_0, \dots X_N$ given by the
diagrams~\eqref{eq:spherediagpushout}.
\begin{equation}\label{eq:spherediagpushout}
\vcenter{ \xymatrix@C=0.5cm@R=0.7cm{
 & & S^{2N+1} & & \\
 X_i \ar@{-->}[urr] &
 V_i \ar[l]^-{\phi_i}\ar@{^{(}->}[ur] & &
 V_j\ar[r]_-{\phi_j}\ar@{_{(}->}[ul] &
 X_j\ar@{-->}[ull]\\
 X_{i,j} \ar@{^{(}->}[u] & &
 V_i\cap V_j \ar@{_{(}->}[ul] \ar@{^{(}->}[ur] \ar[ll]_-{\phi_{ij}} \ar[rr]^-{\phi_{ji}} & &
 X_{i,j} \ar@<0.5ex>@{_{(}->}[u]
}}
\end{equation}
So if $\sim$ is the equivalence relation on the disjoint union $\coprod_i X_i$ generated
by \mbox{$\phi_i(d) \sim \phi_j(d)$} for all $d \in V_i \cap V_j$ and all $i<j$, then
$S^{2N+1} \cong \big(\coprod_i X_i\big)/{\sim}$. (Note that
$\phi_{ji}\circ\phi_{ij}^{-1}=\id_{X_{i,j}}$.)

To motivate our definition of Heegaard quantum spheres later on, we dualize this
multipushout picture of $S^{2N+1}$ to obtain a multipullback presentation of
$C(S^{2N+1})$. Let $\operatorname{res} : C(D) \to C(\TT)$ be the restriction map. For $i
< j$, we write
\begin{multline*}
\pi^i_j :
C(D)^{\otimes i}\otimes  C(\mathbb{T})\otimes C(D)^{{\otimes N-i}}\\
\longrightarrow
 C(D)^{\otimes i}\otimes C(\mathbb{T})\otimes C(D)^{\otimes j-i-1} \otimes C(\mathbb{T})\otimes
C(D)^{\otimes N-j}
\end{multline*}
for the  surjection $\id^{\otimes j} \otimes \operatorname{res} \otimes \id^{\otimes
N-j}$. Then $C(S^{2N+1})$ is naturally isomorphic to
\begin{equation*}
\big\{(f_0, \dots, f_N) \in \mbox{$\bigoplus^N_{i=0}$}
C(D)^{\otimes i}\otimes C(\mathbb{T})\otimes C(D)^{\otimes {N-i}}
   \mathbin{\mid} \pi^i_j(f_j) = \pi^j_i(f_i)\;\text{ for all }i < j\}.
\end{equation*}

\subsection{Gauging diagonal actions and coactions}\label{sec:gauging}

Throughout this paper, we denote a right action of a group $G$ on a space $X$ by
juxtaposition, that is $(x,g) \mapsto xg$. The general idea for converting between
diagonal and rightmost actions of a group $G$ on a topological space $X$ is as follows.
We regard $X\times G$ as a right $G$-space in two different ways, which we distinguish
notationally as follows.
\begin{itemize}
\item We write $(X\times G)^R$ for the  product $X\times G$ with $G$-action
    \mbox{$(x,g)\cdot h := (x, gh)$}.
\item We write $X\times G$ for the same space with diagonal $G$-action
$(x, g) h = (x h,
    gh)$.
\end{itemize}
There is a $G$-equivariant homeomorphism $\kappa : (X\times G)^R \to X\times G$
determined by $\kappa(x, g) := (x g, g)$, with inverse given by $\kappa^{-1}(x,g) = (x g^{-1},g)$.
In general, given any cartesian product of $G$-spaces, we will regard it as a $G$-space
with the diagonal action, except for those of the form $(X \times G)^R$ just described.

In what follows, the tensor product means completed tensor product, and we use the
Heynemann-Sweedler notation (with the summation sign suppressed) for this completed
product. Since all $C^*$-algebras that we tensor are nuclear, this completion  is unique.
We often identify the unit circle $\mathbb{T}$ with the unitary group $U(1)$, and use the
quantum group structure on~$C(U(1))$. Even though we only use the classical compact
Hausdorff group $U(1)$, we are forced to use the quantum-group language of coactions,
etc., to write explicit formulas, and carry out computations.

Let $G$ be a locally compact group, and let $H := C(G)$. Then $S\colon H\to H$, given by
$S(h)(g):=h(g^{-1})$, is the antipode map, $\varepsilon(h):=h(e)$ defines the counit ($e$ is the neutral
element of~$G$), and
\begin{gather*}
\Delta\colon H\longrightarrow H\otimes H\cong C(G\times G), \\
\Delta(h)(g_1,g_2):=h(g_1g_2)=:(h_{(1)}\otimes h_{(2)})(g_1,g_2)= h_{(1)}(g_1)
h_{(2)}(g_2),
\end{gather*}
is a coproduct. If $\alpha:G\to \mathrm{Aut}(A)$ is a $G$-action on a unital $C^*$-algebra $A$, then there is a coaction
$\delta\colon A\rightarrow A\otimes H\cong C(G,A)$ given by
\begin{equation*}
\delta(a)(g):=\alpha_g(a)=:(a_{(0)}\otimes a_{(1)})(g) = a_{(0)}a_{(1)}(g).
\end{equation*}

Consider $A\otimes H$ as a $C^*$-algebra with the diagonal coaction
$$
p\otimes h\longmapsto
p\sw{0}\otimes h\sw{1}\otimes p\sw{1}h\sw{2},
$$
 and denote by $(A\otimes H)^R$ the same
$C^*$-algebra with the coaction on the rightmost factor: $p\otimes h\mapsto p\otimes
h\sw{1}\otimes h\sw{2}$. Then the following
map is a $G$-equivariant (i.e., intertwining the coactions) isomorphism of $C^*$-algebras:
\begin{equation}\label{kappa}
\widehat{\kappa}:(A\otimes H)\longrightarrow (A\otimes H)^R,\quad a\otimes h\longmapsto a\sw{0}\otimes a\sw{1}h.
\end{equation}
\noindent
Its inverse is explicitly given  by
\begin{equation}\label{kappa-1}
\widehat{\kappa}^{-1}:(A\otimes H)^R\longrightarrow (A\otimes H),
\quad a\otimes h\longmapsto a\sw{0}\otimes S(a\sw{1})h.
\end{equation}

\subsection{Affine closed coverings of complex projective spaces}

The odd sphere $S^{2N+1}$ is a $U(1)$-principal bundle. The diagonal action of $U(1)$ on
$S^{2N+1}$ is given by
\begin{equation*}
(z_0,\ldots,z_N) \lambda := (z_0\lambda,\ldots, z_N\lambda).
\end{equation*}
Since $\TT \subseteq D$ is rotation-invariant, this action restricts to a $U(1)$-action
on each \mbox{$D^1\! \times\! \TT\! \times\! D^{N-1}$}, so the multipushout given
by~\eqref{eq:spherediagpushout} is $U(1)$-equivariant.

To obtain a multipushout presentation of $\mathbb{P}^N(\CC) = S^{2N+1}/{U(1)}$, we need
to gauge the diagonal actions to actions on the rightmost components. This will yield an
alternative multipushout presentation of~$S^{2N+1}$. Using the notation of
Section~\ref{sec:gauging}, we write \mbox{$\kappa : (D^N\times U(1))^R \to D^N \times
U(1)$} for the gauging homeomorphism. Identify $U(1)$ with~$\mathbb{T}$, and write
$F_{i,N} : D^N \times U(1) \to D^i \times \mathbb{T} \times D^{N-i}$ for the map
\begin{equation*}
F_{i,N}(d_0, \dots,d_{i-1}, d_i, d_{i+1} \dots, d_{N-1}, d_N)
    := (d_0, \dots,d_{i-1}, d_N, d_{i+1} \dots, d_{N-1}, d_i).
\end{equation*}
We obtain a $U(1)$-equivariant homeomorphism
\begin{equation*}
h_i := F_{i,N} \circ \kappa : (D^N\times U(1))^R \longrightarrow
    D^i \times \mathbb{T}\times D^{N-i}.
\end{equation*}

Let $X^R_i := (D^{N}\times U(1))^R$ for all~$i$. For $i < j < N$, let
\begin{gather*}
X^R_{i,j} := \big(D^{i} \times \mathbb{T} \times D^{N-i-1} \times U(1)\big)^R,\quad
X^R_{j,i} := \big(D^{j-1} \times \mathbb{T} \times D^{N-j} \times U(1)\big)^R,\nonumber\\
\text{and }X_{i,j}:= D^i\times \mathbb{T}\times D^{j-i-1}\times \mathbb{T}\times D^{N-j} =: X_{j,i}.
\end{gather*}
For $i \not= j$, we define $h_{ij} := h_i\bigl|_{X^R_{j,i}} : X^R_{j,i} \to X_{i,j} =
X_{j,i} $.

We use the $h_i$ and $h_{ij}$ to transform the multipushout structure of $S^{2N+1}$
described by~\eqref{eq:spherediagpushout}. Explicitly, for $0\leq i<j\leq N$, we obtain
the commuting diagram~\eqref{eq:wings}.
\begin{equation}\label{eq:wings}
\vcenter{
\xymatrix{
X^R_i\ar[r]^-{h_i}& X_i&&X_j
&X^R_j\ar[l]_-{h_j}\\
X^R_{j,i}\ar[rrd]^-{h_{ij}}\ar@{_{(}->}[u]&&
&&
X^R_{i,j}\ar[lld]_-{h_{ji}}\ar@{_{(}->}[u]\\
&&
{}\save[]+<-0cm,-0.2cm>*\txt<25pc>{
$X_{i,j}$} \restore \phantom{aaaaa} \ar@<0.4ex>@{_{(}->}[uul] \ar@<-0.4ex>@{^{(}->}[uur]
&&
}}
\end{equation}
\noindent
For $i < j$, we define $
     \chi_{ij} := h_{ji}^{-1}\circ h_{ij} : X^R_{j,i} \to X^R_{i,j}
$. (Note that, unlike in the previous multipushout presentation of $S^{2N+1}$, these maps are not identities.)
With this notation, $S^{2N+1}$ is homeomorphic to the quotient of the disjoint union
\begin{equation*}
\coprod_{0\leq i\leq N}(D^{N}\times U(1))^R = \coprod_{0\leq i\leq N} X^R_i
\end{equation*}
by the smallest equivalence relation such that $d \sim \chi_{ij}(d)$ for all $d \in
X^R_{j,i}$. The equivalence relation $\sim$ respects the $U(1)$-actions, so that we
obtain a multipushout presentation of $S^{2N+1}/U(1)\cong\mathbb{P}^N(\mathbb{C})$ by
everywhere restricting $U(1)$ to a point. This multipushout presentation of the complex
projective space agrees with the multipushout presentation used in
\cite[Section~1.2]{hkz12} to obtain the multipullback noncommutative deformation
of~$\mathbb{P}^N(\mathbb{C})$.

\section{Twisted multipullback quantum odd spheres}\label{sec:twisted spheres}

\subsection{Twisted quantum even balls}

Recall that we regard the Toeplitz algebra $\Tt$ as the quantum-disc $C^*$-algebra~\cite{KlimekLesniewski:JFA1993}. Let
$s$ be the generating isometry in $\Tt$ \cite{c-la67,c-la69} and $u$ the generating unitary in $C(\TT)$. Let
$\sigma : \Tt \to C(\TT)$, $s \mapsto u$, denote the symbol map.
 We use the exact sequence
\begin{equation*}
0 \longrightarrow \Kk \longrightarrow \Tt \stackbin{\sigma}{\longrightarrow} C ( \TT ) \longrightarrow 0
\end{equation*}
to regard the circle $\mathbb{T}$ as the boundary of the quantum disc, or two-dimensional quantum ball.
Thus the one-dimensional quantum sphere then corresponds to the quotient $\Tt / \Kk$.
From this perspective, $\Tt^{\otimes N}$ can be regarded as the algebra of a Cartesian
product of $N$ \mbox{two-dimensional} balls, and therefore as a copy of a
$2N$-dimensional (non-round) quantum ball. The quotient $\Tt^{\otimes N+1}/ \Kk^{\otimes
N+1}$ is then viewed as the algebra of the boundary of the quantum ball, that is, a
quantum sphere of dimension $2N+1$. In the same spirit, $\Tt^{\otimes N} \otimes C ( \TT
)$ is regarded as the algebra of the Cartesian product of a $2N$-ball and a circle, which
is to say a $(2N+1)$-dimensional noncommutative solid torus.

By analogy with the
Heegaard splitting of $S^{2N+1}$ in the preceding section, we define the algebra
$C(S^{2N+1}_H)$ of continuous functions on the Heegaard quantum sphere as a multipullback
of the $C^*$-algebras $\Tt^{\otimes i} \otimes C(\mathbb{T}) \otimes \Tt^{\otimes N-i}$
with respect to the maps
\begin{equation*}
\quad\pi^i_j\colon \Tt^{\otimes i} \otimes C(\mathbb{T}) \otimes \Tt^{\otimes N-i}\longrightarrow
\Tt^{\otimes i} \otimes C(\mathbb{T}) \otimes \Tt^{\otimes j-i-1}\otimes C(\mathbb{T})\otimes\Tt^{\otimes N-j},
\quad i<j,
\end{equation*}
given by $\pi^i_j := \id_{\Tt^{\otimes i} \otimes C(\TT) \otimes \Tt^{j-i-1}} \otimes
\sigma \otimes \id_{\Tt^{\otimes N-j}}$.

In Section~\ref{untwisted}, we will realize $C(S^{2N+1}_H)$ as the special case where
$\theta = 0$ of a multipullback of twisted tensor products of the same sort. We begin by
defining the twisted Toeplitz algebras $\Tt^{N+1}_\theta$, which we view as
twisted-quantum-ball $C^*$-algebras.

\begin{dfn}
Fix $N > 0$, and suppose that $\theta = (\theta_{ij})^N_{i,j = 0} \in M_{N+1}(\RR)$ is
antisymmetric in the sense that $\theta_{ij} = -\theta_{ji}$. We define the twisted
Toeplitz algebra $\Tt^{N+1}_\theta$ to be the universal $C^*$-algebra generated by
isometries $\{w_{0}^\theta, \dots, w^\theta_{N}\}$ such that
\begin{equation*}
w_{j}^\theta w_{k}^\theta = e^{2\pi i \theta_{jk}} w_{k}^\theta w_{j}^\theta
    \quad\text{ and }\quad
w_{j}^{\theta*} w_{k}^\theta = e^{-2\pi i \theta_{jk}} w_{k}^\theta
    w^{\theta *}_{j}\quad\text{ for all $j \not= k$.}
\end{equation*}
\end{dfn}

With this in hand, we are ready to present our definition of the twisted Heegaard quantum
sphere~$S^{2N+1}_{H, \theta}$, which we view as the boundary of a twisted quantum ball.
Thus we generalize the 3-dimensional case $S^{3}_{H, \theta}$ introduced and analyzed
in~\cite{bhms05}.
\begin{dfn}\label{dfn:twisted sphere}
For $0 \le i \le N$, let $I_i^\theta$ denote the ideal of $\Tt^{N+1}_\theta$ generated by
$1 - w_{i}^\theta w^{\theta*}_{i}$, and for $i \not= j$, let $I_{ij}^\theta := I_i^\theta
+ I_j^\theta$. Let $B_i^\theta := \Tt^{N+1}_\theta/I_i^\theta$ and $B_{ij}^\theta :=
\Tt^{N+1}_\theta / I_{ij}^\theta$. Also, let
\begin{equation}\label{sigmaitheta}
\sigma_i : \Tt^{N+1}_\theta \longrightarrow B_i^\theta\quad\text{ and }\quad \pi^i_j : B_i^\theta \longrightarrow B_{ij}^\theta
\end{equation}
\noindent be the natural quotient maps. We define the \emph{twisted Heegaard quantum
sphere} $C^*$-algebra as the multipullback of the algebras $B_i^\theta$ over the
homomorphisms $\pi^i_j$, that is
\begin{equation*}\textstyle
C(S^{2N+1}_{H, \theta}) := \Big\{(b_0, \dots, b_N) \in \bigoplus^N_{i=0} B_i^\theta \;\Big|\; \pi^i_j(b_i)
= \pi^j_i(b_j)\text{ for all } 0 \le i < j \le N\Big\}.
\end{equation*}
\end{dfn}

To ease notation we define $w_k^{\theta;i} := \sigma_i(w^\theta_k)$ and $w_k^{\theta;ij}
:= w^\theta_k+I_i^\theta+I_j^\theta$ for all $k$ and distinct $i,j$. We define
\mbox{$\mathbf{s}_i \in C(S^{2N+1}_{H,\theta})$} by
\begin{equation*}
\mathbf{s}_i := \big(w^{\theta;0}_i, \dots, w^{\theta;N}_i\big).
\end{equation*}
For $i,j\in\{0,\ldots,N\}$, we have
\begin{gather}
\mathbf{s}_i\mathbf{s}_j=e^{2\pi i\theta_{ij}}\mathbf{s}_j\mathbf{s}_i,\quad
    \mathbf{s}_i\mathbf{s}_j^*=e^{-2\pi i\theta_{ij}}\mathbf{s}_j^*\mathbf{s}_i,\quad\text{when } i\neq j,\nonumber\\
\mathbf{s}_i^*\mathbf{s}_i=1,\quad\text{ and}\nonumber\\
\prod_{k=0}^N(1-\mathbf{s}_k\mathbf{s}_k^*)=0.\label{prodprop}
\end{gather}

The universal property of $\Tt^{N+1}_\theta$ yields a $U(1)^{N+1}$-action satisfying
$(\lambda_0, \dots, \lambda_N)\cdot w^\theta_j = \lambda_jw^\theta_j$. We call this the
\emph{gauge action} on $\Tt^{N+1}_\theta$. This action descends to each $B_i$ and each
$B_{ij}$, and hence induces a $U(1)^{N+1}$-action on $C(S^{2N+1}_{H, \theta})$, also
called the gauge action. Restricting to the diagonal in $U(1)^{N+1}$ gives a
$U(1)$-action $\alpha$ on $C(S^{2N+1}_{H, \theta})$ such that
\begin{equation} \label{eq:alphadef}
\alpha_\lambda(b_0, \dots, b_N) = (\lambda \cdot b_0, \dots, \lambda \cdot b_N).
\end{equation}

\subsection{A universal presentation}

We prove, using Whitehead's twisted relative Cuntz--Krieger algebras of higher-rank
graphs \cite{BenThesis} (see also \cite{SimsWhiteheadEtAl:xx13}), that the twisted
Heegaard quantum sphere $C^*$-algebra of Definition~\ref{dfn:twisted sphere} enjoys a
universal property.

\begin{thm}\label{prp:pullback isomorphism}
Consider an integer $N \ge 1$ and a antisymmetric matrix $\theta \in M_{N+1} (\RR)$. Let
$A_\theta(N+1)$ be the universal $C^*$-algebra generated by isometries $s_0, \dots, s_N$
satisfying
\begin{equation}\label{eq:commutation}
s_i s_j = e^{2\pi i \theta_{ij}} s_j s_i\quad\text{ and }\quad
s_i s^*_j = e^{-2\pi i\theta_{ij}} s^*_j s_i,
\end{equation}
and the sphere equation
\begin{equation}\label{eq:sphere}
\prod^N_{i=0} (1 - s_i s^*_i) = 0.
\end{equation}
\noindent Then there is a $U(1)$-action on $A_\theta(N+1)$ such that $\lambda \cdot s_i =
\lambda s_i$ for all $i$, and there is a $U(1)$-equivariant isomorphism $\phi_\theta :
A_\theta(N+1) \to C(S^{2N+1}_{H,\theta})$ such that
\begin{equation*}
\phi_\theta(s_i) = \mathbf{s}_i = \big(w^{\theta;0}_i, \dots, w^{\theta;N}_i\big) \quad\text{ for all $i$.}
\end{equation*}
\noindent Furthermore, the maps $\pi^i_j : B_i \to B_{ij}$ satisfy the cocycle condition
of Definition~\ref{def:cocycle}.
\end{thm}

The existence of the $U(1)$-action on $A_\theta(N+1)$ and of the homomorphism
$\phi_\theta$ follows from the universal property of $A_\theta  (N+1)$. We use the
technology of twisted relative higher-rank graph $C^*$-algebras
\cite{SimsWhiteheadEtAl:xx13} to see that $\phi_\theta$ is injective. For surjectivity,
and to see that the cocycle condition is satisfied, we will need the following technical
lemma.

\begin{lem}\label{lem:quotient surjective}
Let $A$ be a $C^*$-algebra and suppose that $I_0, \dots, I_n$ are ideals of $A$. Suppose
that $a_0, \dots, a_n \in A$ satisfy $a_i + (I_i + I_j) = a_j + (I_i + I_j)$ for all
$i,j$. Then there exists $a \in A$ such that $a + I_i = a_i + I_i$ for all $i$.
\end{lem}
\begin{proof}
We proceed by induction on $n$. The base case $n = 0$ is trivial. Suppose as an inductive
hypothesis that there exists $a' \in A$ such that $a' + I_i = a_i + I_i$ for all $i < n$.
Then $a' + (I_i + I_n) = a_n + (I_i + I_n)$ for all $i < n$, whence
\begin{equation}\label{eq:a'-an}\textstyle
a' - a_n \in \bigcap_{i<n} (I_i + I_n).
\end{equation}
Since the ideals of the $C^*$-algebra $A$ form a distributive lattice with meet given by
intersection and join given by sum, we have
\begin{equation*}
\bigcap_{i<n} (I_i + I_n)
	= \Big(\bigcap_{i<n} I_i\Big) + \sum_{\emptyset \not= F \subseteq \{0, \dots, n-1\}}
							\Big(I_n \cap \bigcap_{i \not\in F} I_i\Big)
	\subseteq \Big(\bigcap_{i<n} I_i\Big) + I_n.
\end{equation*}
Combining this with~\eqref{eq:a'-an}, we obtain $a' - a_n = b' - b_n$, where $b' \in
\bigcap_{i=0}^{n-1} I_i$ and $b_n \in I_n$. Put $a := a' - b'$. Since $b' \in I_i$ for
all $i \le n-1$, we have $a + I_i = a' + I_i = a_i + I_i$ for $i \le n-1$. Furthermore,
$a = a' - b' = a_n - b_n$ and $b_n \in I_n$, so $a + I_n = a_n + I_n$ too.
\end{proof}

To prove Theorem~\ref{prp:pullback isomorphism}, we use twisted higher-rank graph
$C^*$-algebras. The general theory of these objects requires significant background, but
fortunately the only higher-rank graphs we need to consider are the following elementary
examples.

Let $\Lambda$ denote a copy of the monoid $\NN^{N+1}$ under addition. This becomes an
$(N+1)$-graph in the sense of \cite[Definition~1.1]{KumjianPask:NYJM2000} under the
degree map $d : \Lambda \to \NN^{N+1}$ given by the identity map on $\NN^{N+1}$. We write
$e_0,\dots, e_N$ for the canonical generators of $\NN^{N+1}$. Since we are viewing
$\Lambda$ as a category, we write $\mu\nu$ for the composition of elements $\mu,\nu$.
This is really just $\mu + \nu$ when the two are regarded as elements of $\NN^{N+1}$. The
unique vertex of $\Lambda$ is $0 \in \NN^{N+1}$. For $\mu = (\mu_0, \dots, \mu_N) \in
\Lambda$, we write $|\mu| := \sum^N_{i=0} \mu_i$. A \emph{cocycle} on $\Lambda$ is a map
$c : \Lambda \times \Lambda \to \TT$ satisfying the cocycle identity
$c(\mu,\nu)c(\lambda,\mu\nu) = c(\lambda,\mu)c(\lambda\mu,\nu)$ for all $\lambda,\mu,\nu
\in \Lambda$. Since $\NN^{N+1}$ is directed, every finite $F \subseteq \Lambda \setminus
\{0\}$ is exhaustive as in \cite[Section~2]{SimsWhiteheadEtAl:xx13}. So given any
collection $\Ee$ of finite subsets of $\Lambda \setminus \{0\}$, we can form the twisted
relative Cuntz--Krieger algebra $C^*(\Lambda, c; \Ee)$, which is generated by isometries
$\{s^c_\Ee(\lambda) : \lambda \in \Lambda\}$ satisfying relations (TCK1)--(TCK4)~and~(CK)
of \cite[Section~3]{SimsWhiteheadEtAl:xx13}.

\begin{lem}\label{lem:k-graph alg}
Let $\Lambda$ denote $\NN^{N+1}$ regarded as an $(N+1)$-graph as above. Fix a
antisymmetric matrix $\theta \in M_{N+1}(\RR)$. There is a cocycle $c$ on $\Lambda$ given
by
\begin{equation} \label{eq:cdef}
    c(\mu,\nu) := e^{\pi i (d(\mu)^T\! \theta d(\nu))}.
\end{equation}
Let $\Ee := \{\{e_0, \dots, e_N\}\}$. Then there is an isomorphism $A_\theta(N+1) \to
C^*(\Lambda, c; \Ee)$ that carries $w_i \in A_\theta(N+1)$ to $s^c_\Ee(e_i) \in
C^*(\Lambda, c; \Ee)$ for $0 \le i \le N$.
\end{lem}
\begin{proof}
One checks that $A_\theta(N+1)$ and $C^*(\Lambda, c;\Ee)$ have the same universal
property.
\end{proof}

\begin{proof}[Proof of Theorem~\ref{prp:pullback isomorphism}]
The relations \eqref{eq:commutation}~and~\eqref{eq:sphere} are invariant under
multiplication of  the $s_i$ by any fixed $\lambda \in U(1)$. Thus the universal property
of $A_\theta(N+1)$ yields the desired $U(1)$-action.

The universal property of $\Tt^{N+1}_\theta$ yields a homomorphism
\begin{equation*}
\psi_\theta : \Tt^{N+1}_\theta \longrightarrow C(S^{2N+1}_{H,\theta})
\quad\text{ given by }\quad\psi_\theta(a) = (\sigma_0(a), \sigma_1(a),
\dots, \sigma_N(a)).
\end{equation*}
Applying Lemma~\ref{lem:quotient surjective} to $A = \Tt^{N+1}_\theta$ and the ideals
$I_i = \ker(\sigma_i)$ shows that
\begin{equation*}
C(S^{2N+1}_{H,\theta}) = \{(\sigma_0(a), \sigma_1(a), \dots, \sigma_N(a)) \mid a \in \Tt^{N+1}_\theta\},
\end{equation*}
so that $\psi_\theta$ is surjective. Since $\prod^N_{j=0} (1 - w_j w^*_j) \in
\ker\sigma_i$ for each $i$, it belongs to $\ker\psi_\theta$, so $\psi_\theta$ descends to
a surjective homomorphism $\phi_\theta : A_\theta(N+1) \to C(S^{2N+1}_{H,\theta})$ such
that $\phi_\theta(s_i) = \mathbf{s}_i$ for all $i$.

By Lemma~\ref{lem:k-graph alg} it suffices to show that the homomorphism
$$
\rho :C^*(\Lambda, c; \Ee) \longrightarrow C(S^{N+1}_{H,\theta})
$$
satisfying $\rho(s^c_\Ee(e_i)) =
\phi_\theta(s_i)
$
 is injective. For this, we aim to apply the gauge-invariant uniqueness
theorem \cite[Theorem~3.15]{SimsWhiteheadEtAl:xx13} for $C^*(\Lambda, c; \Ee)$.

The homomorphism $\rho$ is equivariant for the gauge actions on $C(S^{2N+1}_{H,\theta})$
and $C^*(\Lambda, c; \Ee)$. Since $s^c_\Ee(0)$ is the identity element of $C^*(\Lambda,
c; \Ee)$, we have
$$
\rho(s^c_\Ee(0)) = (1, 1, \dots, 1) \not= 0.
$$
 Hence, by
\cite[Theorem~3.15]{SimsWhiteheadEtAl:xx13}, it suffices to show that for each finite $F$
in the complement of the satiation $\overline{\Ee}$ of $\Ee$ (see
\cite[page~837]{SimsWhiteheadEtAl:xx13}),
$$
\rho\big(\prod_{\mu \in F} (s^c_\Ee(0) -
s^c_\Ee(\mu)s^c_\Ee(\mu)^*)\big) \not= 0.
$$

The set
\begin{multline*}
\Ee' := \{F \subset \Lambda \setminus \{0\} \mid\\ \text{ there exists $i > 0$ such that
    $|p| > i$ implies $p \ge q$ for some $q \in F$}\}
\end{multline*}
satisfies (S1)--(S4) on page~87 of \cite{SimsWhiteheadEtAl:xx13} and contains $\Ee$. An
induction shows that any set containing $\Ee$ and satisfying (S1)--(S4) contains $\Ee'$.
So $\Ee' = \overline{\Ee}$. So for a finite set $F \not\in \overline{\Ee}$, there is a
sequence $(p^i)$ in $\Lambda$ with $|p^i| \to \infty$ such that $p^i \not\ge q$ for all
$q \in F$ and all $i \in \NN$. By passing to a subsequence, we may assume that $p^i_j \to
\infty$ for some $j \le N$. Since $p^i \not\ge q$ for all $q \in F$ and all $i$, it
follows that $q \in F$ implies $q_l > 0$ for some $l \not= j$. Hence there exists $l
\not= j$ such that $q \ge e_l$, which forces
\begin{equation*}
s^c_{\Ee}(q)s^c_\Ee(q)^*  = s^c_\Ee(e_l) s^c_\Ee(q - e_l) s^c_\Ee(q - e_l)^* s^c_Ee(e_l)^* \le s^c_{\Ee}(e_l)s^c_\Ee(e_l)^*.
\end{equation*}
Thus
\begin{equation*}
\rho\big(1 - s^c_{\Ee}(q)s^c_\Ee(q)^*\big) \ge \rho\big(1 - s^c_{\Ee}(e_l)s^c_\Ee(e_l)^*\big) = 1 - \mathbf{s}_l \mathbf{s}^*_l.
\end{equation*}

Applying this reasoning to each $q \in F$, we obtain
\begin{equation*}
\rho\Big(\prod_{q \in F} (1 - s^c_\Ee(q) s^c_\Ee(q)^*)\Big) \ge \prod_{l \not= j} (1 - \mathbf{s}_l \mathbf{s}^*_l).
\end{equation*}

Since each $\mathbf{s}_l \in C(S^{2N+1}_{H,\theta}) \subseteq \bigoplus^N_{i=0} B_i$
(where $B_i = \Tt^{N+1}_\theta/I_i$), the $j$th coordinate of $\prod_{l \not= j} (1 -
\mathbf{s}_l \mathbf{s}^*_l)$ is
\begin{equation}\label{eq:product}
\Big(\prod_{l \not= j} (1 - \mathbf{s}_l \mathbf{s}^*_l)\Big)_j
    = \sigma_j\Big(\prod_{l \not= j} (1 - w_l w^*_l)\Big).
\end{equation}
So it suffices to show that the right-hand side of~\eqref{eq:product} is nonzero. Since
$\sigma_j(\Tt^{N+1}_\theta)$ is universal for the same relations as the twisted relative
Cuntz--Krieger algebra $C^*(\Lambda, c; \{e_j\})$, there is an isomorphism
$\sigma_j(\Tt^{N+1}_\theta) \to C^*(\Lambda, c; \{e_j\})$ that carries $\sigma_j(w_l)$ to
$s^c_{\{e_j\}}(e_l)$ for each~$l$. The satiation $\overline{\{e_j\}}$ of $\{e_j\}$ does
not contain the set $\{e_l \mid l \not= j\}$, so
\cite[Proposition~3.9]{SimsWhiteheadEtAl:xx13} implies that
\begin{equation*}
\prod_{l \not= j} (1 - s^c_{\{e_j\}}(e_l) s^c_{\{e_j\}}(e_l)^*) \not = 0,
\end{equation*}
giving $\sigma_j\Big(\prod_{l \not= j} (1 - w_l w^*_l)\Big) \not= 0$ as required. This
completes the proof that $\phi_\theta$ is an isomorphism.

Since each $B_i = \Tt^{N+1}_\theta/I_i$ and $B_{ij} = \Tt^{N+1}_\theta/(I_i + I_j)$ by
definition, the homomorphisms $\pi^i_j$ are distributive in the sense of
\cite[Definition~2]{HZ}. Lemma~\ref{lem:quotient surjective} shows in particular that
given distinct $i,j,k$ and elements $b_i \in B_i$ and $b_j \in B_j$ such that
$\pi^i_j(b_i) = \pi^j_i(b_j)$, there exists $b_k \in B_k$ such that $\pi^k_i(b_k) =
\pi^i_k(b_i)$ and $\pi^k_j(b_k) = \pi^j_k(b_j)$. Hence Theorem~1 of \cite{HZ} implies
that the $\pi^i_j$ satisfy the cocycle condition of Definition~\ref{def:cocycle}.
\end{proof}

\subsection{Strong connections}

Since we focus on free $U(1)$-actions on unital 
 $C^*$-algebras, we avoid the general
coalgebraic formalism of strong connections of~\cite{bh04}, and formulate the concept of
a strong connection from~\cite{h-pm96} solely for $U(1)$-actions on unital
$C^*$-algebras.

Let $A$ be a unital $C^*$-algebra carrying a $U(1)$-action. For $m \in \ZZ$, recall that
$A_m$ denotes the spectral subspace $\{a \in A \mid \lambda \cdot a = \lambda^m a\text{
for all } \lambda \in U(1)\}$. We write $\CC[u, u^*]$ for the $^*$-algebra of Laurent
polynomials. Let $\ell$ be a unital linear map
\begin{equation*}\textstyle
\ell:\CC[u, u^*] \longrightarrow \big(\bigoplus_{m \in \ZZ} A_m\big) \;
\mathbin{\underset{\mathrm{alg}}{\otimes}}\; \big(\bigoplus_{m \in \ZZ} A_m\big)\subseteq A\underset{\mathrm{alg}}{\otimes} A,
\end{equation*}
where $\bigoplus_{m \in \ZZ} A_m$ denotes the algebraic direct sum of the spectral
subspaces. We say that $\ell$ is a \emph{strong connection} for the $U(1)$-action on $A$
if, writing
$$
m_A : A \mathbin{\otimes_{\mathrm{alg}}} A \longrightarrow A
$$
 for the multiplication
map, we have
\begin{equation}\label{eq:m circ l}
(m_A \circ \ell)(h) = h(1)1_A\quad\text{ for all $h \in \CC[u, u^*]$},
\end{equation}
and
\begin{equation}\label{eq:degree}
\ell(u^n) \in A_{-n} \otimes A_n \quad\text{ for all $n \in \ZZ$.}
\end{equation}

By~\cite{u-kh81} the existence of a strong connection is equivalent to strong grading:
\begin{equation*}
A_m A_n=A_{m+n}\quad\text{for all } m,n\in\ZZ.
\end{equation*}
Moreover, by the main theorem of~\cite{bdh} combined with \cite[Theorem~2.5(1)]{bh04},
the existence of a strong connection is
equivalent to freeness.

\subsubsection{A strong connection on \texorpdfstring{$S^{2N+1}_{H,\theta}$}{the twisted quantum sphere}}\label{3.2}

In what follows, we will also need the following family of $U(1)$-fixed elements of
$C(S^{2N+1}_{H,\theta})$:
\begin{equation*}
H_N=1,\quad H_i=\prod_{j=i+1}^N(1-\mathbf{s}_j\mathbf{s}_j^*),\quad i\in\{0,\ldots,N-1\}.
\end{equation*}

Consider the linear map
\begin{equation*}\textstyle
\ell: \CC[u, u^*] \longrightarrow \big(\bigoplus_{m \in \ZZ} C(S^{2N+1}_{H,\theta})_m\big)
    \;\mathbin{\underset{\mathrm{alg}}{\otimes}}\; \big(\bigoplus_{m \in \ZZ} C(S^{2N+1}_{H,\theta})_m\big)
\end{equation*}

\noindent defined inductively as follows:
\begin{gather}
\ell(1) := 1 \otimes 1,\qquad \ell(u^n) = \mathbf{s}^{*n}_0 \otimes \mathbf{s}^n_0\text{ for $n > 0$,\quad and}\nonumber\\
\ell(u^{n-1}):=\sum_{0\leq k\leq N}\Big((\mathbf{s}_k \otimes 1)\ell(u^n)(1 \otimes
\mathbf{s}^*_kH_{k})\Big)\text{ for $n \le 0$.}\label{mystrong2}
\end{gather}
Then $\ell$ is a strong connection for the $U(1)$-action on $C(S^{2N+1}_{H,\theta})$:
Equation~\eqref{eq:degree} for $n \ge 0$ is trivial, and for $n < 0$ follows from an
elementary induction argument. Equation~\eqref{eq:m circ l} for $n \ge 0$ is trivial
because $\mathbf{s}_0$ is an isometry. To check it for $n < 0$, we first use the sphere
equation~\eqref{prodprop} to see that $\sum_{k=0}^N \mathbf{s}_k \mathbf{s}_k^* H_k = 1$,
and then employ a straightforward induction argument (see the proof of
\cite[Lemma~4.2]{hms06}) using the recursive formula~\eqref{mystrong2}.

\section{Twisted multipullback quantum complex projective spaces}\label{sec:projective}
\noindent
Our twisted multipullback quantum odd sphere $C^*$-algebras (see
Definition~\ref{dfn:twisted sphere}) yield a natural construction of a family of
$\theta$-twisted complex projective space $C^*$-algebras as fixed-point algebras. Using
the $U(1)$-action $\alpha$ on $C(S^{2N+1}_{H, \theta})$ from
equation~\eqref{eq:alphadef}, we define
\begin{equation*}
    C(\mathbb{P}^N_\theta(\Tt)) := C(S^{2N+1}_{H, \theta})^\alpha .
\end{equation*}
To study $C(S^{2N+1}_{H,\theta})^\alpha$, we gauge the diagonal action $\alpha$ on
$C(S^{2N+1}_{H,\theta})$ to an action on a single twisted component, where it is easy to
determine the $U(1)$-invariant subalgebra. As in Section~\ref{sec:twisted spheres},
restricting to the diagonal subgroup of $U(1)^{N+1}$ yields a diagonal action on
$\Tt^{N+1}_\theta$ given by $\lambda\cdot w_j^\theta := (\lambda,\ldots,\lambda)\cdot
w_j^\theta = \lambda w_j^\theta$. We can also compose with the coordinate inclusions
$U(1) \hookrightarrow U(1)^{N+1}$ to obtain actions $\cdot_i$ of $U(1)$ given by
\begin{equation*}
\lambda \cdot_i w_j=\begin{cases}
w_j&\text{if $i\neq j$}\\
\lambda w_i&\text{if $i=j$.}
\end{cases}
\end{equation*}
Since that gauge action descends to the quotients by the $I^\theta_k$ and
$I^\theta_{kj}$, so do these $U(1)$-actions. We will consider $B^\theta_i$ and
$B^\theta_{ij}$ to be endowed with the diagonal $U(1)$-action and we denote by
$B^{\theta;R_k}_i$ and $B^{\theta;R_k}_{ij}$ the same $C^*$-algebras endowed with the
$U(1)$-action on the $k$-th twisted component. Accordingly, we will write the generators
of $B^{\theta;R_k}_i$ and $B^{\theta;R_k}_{ij}$ as $w_l^{\theta;i;R_k}$ and
$w_l^{\theta;ij;R_k}$, respectively.

\begin{lem}
For any $(N+1)\times(N+1)$ antisymmetric real matrix $\theta$ and $0\leq i\leq N$, define
antisymmetric real matrices $\kappa_i(\theta)$ and $\kappa^{-1}_i(\theta)$ of the same
size by
\begin{gather}
\kappa_i(\theta)_{jk}:=\theta_{ij}+\theta_{jk}+\theta_{ki}, \ \text{if\ }j,k\neq i,\quad
\kappa_i(\theta)_{ij}:=\theta_{ij},\nonumber\\
\kappa_i^{-1}(\theta)_{jk}:=-\theta_{ij}+\theta_{jk}-\theta_{ki} \ \text{if\ }jk\neq
i,\quad \kappa_i^{-1}(\theta)_{ik}:=\theta_{ik}.\label{eq:kappa}
\end{gather}
Then
    $\kappa_i^{-1}(\kappa_i(\theta))=\theta=\kappa_i(\kappa_i^{-1}(\theta))$,
 and there exists a $U(1)$-equivariant $C^*$-isomorphism
$
\kappa_i:B^\theta_i\rightarrow B^{\kappa_i(\theta);R_i}_i
$
such that
\begin{gather}
\kappa_i(w_k^{\theta;i}) :=w_k^{\kappa_i(\theta);i;R_i}w_i^{\kappa_i(\theta);i;R_i} \
\text{if\ }i\neq k,\quad \kappa_i(w_i^{\theta;i}):=w_i^{\kappa_i(\theta);i;R_i},
\nonumber\\
\kappa_i^{-1}(w_k^{\kappa_i(\theta);i;R_i}):=w_k^{\theta;i}(w_i^{\theta;i})^* \
\text{if\ }i\neq k,\quad \kappa_i^{-1}(w_i^{\kappa_i(\theta);i;R_i}):=w_i^{\theta;i}. \label{it:C* kappas}
\end{gather}
\end{lem}
\begin{proof}
The equalities $\kappa_i^{-1}(\kappa_i(\theta))=\theta=\kappa_i(\kappa_i^{-1}(\theta))$
follow from elementary calculations using~\eqref{eq:kappa}. To see that~\eqref{it:C*
kappas} defines  *-homomorphisms, note  that, by the universal property of $\Tt^\theta$
and the definition of $I^\theta_{i}$,  it suffices to check that the elements
$\kappa_i(w^{\theta;i}_k)$ and $\kappa_i^{-1}(w^{\kappa_i(\theta);i;R_i}_k)$ satisfy
respectively the relations that determine $B^\theta_i$  and $B^{\kappa_i(\theta);R_i}_i$.
 Let $i,j,k$ be all distinct (the cases where $k=i$ or $j=i$ are trivial).
\begin{enumerate}
\item Since $(w_k^{\theta;i})^*w_k^{\theta;i}=1$, we must have
    $\kappa_i((w_k^{\theta;i})^*w_k^{\theta;i})=1$. Furthermore,
    \begin{multline*}
    \kappa_i((w_k^{\theta;i})^*w_k^{\theta;i})
        =\kappa_i(w_k^{\theta;i})^*\kappa_i(w_k^{\theta;i})\\
        =(w_i^{\kappa_i(\theta);i;R_i})^*(w_k^{\kappa_i(\theta);i;R_i})^* w_k^{\kappa_i(\theta);i;R_i}w_i^{\kappa_i(\theta);i;R_i}
        =1.
    \end{multline*}
\item Since $(w_k^{\kappa_i(\theta);i;R_i})^*w_k^{\kappa_i(\theta);i;R_i}=1$, we must
    have
    $$\kappa_i^{-1}\bigl((w_k^{\kappa_i(\theta);i;R_i})^*w_k^{\kappa_i(\theta);i;R_i}\bigr)
    =1.$$ Furthermore,
    \begin{multline*}
        \kappa_i^{-1}\bigl((w_k^{\kappa_i(\theta);i;R_i})^*w_k^{\kappa_i(\theta);i;R_i}\bigr)
        =\kappa_i^{-1}(w_k^{\kappa_i(\theta);i;R_i})^*\kappa_i^{-1}(w_k^{\kappa_i(\theta);i;R_i})\\
        =w_i^{\theta;i}(w_k^{\theta;i})^* w_k^{\theta;i}(w_i^{\theta;i})^*
        =1
    \end{multline*}
\item Since $w_{j}^{\theta;i} w_{k}^{\theta;i} = e^{2\pi i \theta_{jk}}
    w_{k}^{\theta;i} w_{j}^{\theta;i}$, we must have $$\kappa_i(w_{j}^{\theta;i}
    w_{k}^{\theta;i}) = e^{2\pi i \theta_{jk}} \kappa_i(w_{k}^{\theta;i}
    w_{j}^{\theta;i}).$$ Furthermore,
\begin{align*}
\kappa_i(w_{j}^{\theta;i} w_{k}^{\theta;i})&=
\kappa_i(w_{j}^{\theta;i}) \kappa_i(w_{k}^{\theta;i})\\
&=w_j^{\kappa_i(\theta);i;R_i}w_i^{\kappa_i(\theta);i;R_i}
w_k^{\kappa_i(\theta);i;R_i}w_i^{\kappa_i(\theta);i;R_i}\\
&=e^{2\pi i\kappa_i(\theta)_{ik}}
w_j^{\kappa_i(\theta);i;R_i}
w_k^{\kappa_i(\theta);i;R_i}
w_i^{\kappa_i(\theta);i;R_i}
w_i^{\kappa_i(\theta);i;R_i}\\
&=e^{2\pi i\bigl(\kappa_i(\theta)_{ik}+\kappa_i(\theta)_{jk}\bigr)}
w_k^{\kappa_i(\theta);i;R_i}
w_j^{\kappa_i(\theta);i;R_i}
w_i^{\kappa_i(\theta);i;R_i}
w_i^{\kappa_i(\theta);i;R_i}\\
&=e^{2\pi i\bigl(\kappa_i(\theta)_{ik}+\kappa_i(\theta)_{jk}
+\kappa_i(\theta)_{ji}\bigr)}
w_k^{\kappa_i(\theta);i;R_i}
w_i^{\kappa_i(\theta);i;R_i}
w_j^{\kappa_i(\theta);i;R_i}
w_i^{\kappa_i(\theta);i;R_i}\\
&=e^{2\pi i\bigl(\kappa_i(\theta)_{ik}+\kappa_i(\theta)_{jk}
+\kappa_i(\theta)_{ji}\bigr)}
\kappa_i(w_{k}^{\theta;i})\kappa_i(w_{j}^{\theta;i})\\
&=e^{2\pi i\bigl(\kappa_i(\theta)_{ik}+\kappa_i(\theta)_{jk}
+\kappa_i(\theta)_{ji}\bigr)}
\kappa_i(w_{k}^{\theta;i}w_{j}^{\theta;i}).
\end{align*}
It remains to show  that $\theta_{jk} = \kappa_i(\theta)_{ik} + \kappa_i(\theta)_{jk}
+ \kappa_i(\theta)_{ji}$. Since $\kappa_i(\theta)_{ik}=\theta_{ik}$ and
$\kappa_i(\theta)_{ij}=\theta_{ij}$, we have
$\theta_{jk}=\theta_{ik}+\kappa_i(\theta)_{jk}+\theta_{ji}$, so
$\kappa_i(\theta)_{jk}=\theta_{ij} + \theta_{jk} +\theta_{ki}$ by antisymmetry of
$\theta$.

\item Since $(w_{j}^{\theta;i})^* w_{k}^{\theta;i} = e^{-2\pi i \theta_{jk}}
    w_{k}^{\theta;i} (w_{j}^{\theta;i})^*$, we must have
$$\kappa_i\bigl((w_{j}^{\theta;i})^* w_{k}^{\theta;i}\bigr)
 = e^{-2\pi i \theta_{jk}}
 \kappa_i\bigl(w_{k}^{\theta;i} (w_{j}^{\theta;i})^*\bigr).$$
 Furthermore,
 \begin{align*}
 &\kappa_i\bigl((w_{j}^{\theta;i})^* w_{k}^{\theta;i}\bigr)\\
 &=\kappa_i(w_{j}^{\theta;i})^* \kappa_i(w_{k}^{\theta;i})\\
 &=
 (w_i^{\kappa_i(\theta);i;R_i})^*
 (w_j^{\kappa_i(\theta);i;R_i})^*
 w_k^{\kappa_i(\theta);i;R_i}
 w_i^{\kappa_i(\theta);i;R_i}\\
 &=e^{2\pi i\bigl(-\kappa_i(\theta){jk} - \kappa_i(\theta)_{ji}
 -\kappa_i(\theta)_{ik}\bigr)}
  w_k^{\kappa_i(\theta);i;R_i}
  (w_i^{\kappa_i(\theta);i;R_i})^*
 w_i^{\kappa_i(\theta);i;R_i}
 (w_j^{\kappa_i(\theta);i;R_i})^*\\
 &=e^{2\pi i\bigl(-\kappa_i(\theta){jk} - \kappa_i(\theta)_{ji}
 -\kappa_i(\theta)_{ik}\bigr)}
  w_k^{\kappa_i(\theta);i;R_i}
  w_i^{\kappa_i(\theta);i;R_i}
  (w_i^{\kappa_i(\theta);i;R_i})^*
 (w_j^{\kappa_i(\theta);i;R_i})^*\\
 &=e^{2\pi i\bigl(-\kappa_i(\theta){jk} - \kappa_i(\theta)_{ji}
 -\kappa_i(\theta)_{ik}\bigr)}
 \kappa_i(w_{k}^{\theta;i})\kappa_i(w_{j}^{\theta;i})^*\\
 &=e^{2\pi i\bigl(-\kappa_i(\theta){jk} - \kappa_i(\theta)_{ji}
 -\kappa_i(\theta)_{ik}\bigr)}
 \kappa_i\bigl(w_{k}^{\theta;i}(w_{j}^{\theta;i})^*\bigr)\\
 &=e^{-2\pi i \theta_{jk}}
 \kappa_i\bigl(w_{k}^{\theta;i} (w_{j}^{\theta;i})^*\bigr).
 \end{align*}
 \item Since $w_{j}^{\kappa_i(\theta);i;R_i} w_{k}^{\kappa_i(\theta);i;R_i} =
     e^{2\pi i \kappa_i(\theta)_{jk}} w_{k}^{\kappa_i(\theta);i;R_i}
     w_{j}^{\kappa_i(\theta);i;R_i}$, we must have
$$\kappa_i^{-1}(w_{j}^{\kappa_i(\theta);i;R_i} w_{k}^{\kappa_i(\theta);i;R_i})
= e^{2\pi i \kappa_i(\theta)_{jk}}
\kappa_i^{-1}(w_{k}^{\kappa_i(\theta);i;R_i}
w_{j}^{\kappa_i(\theta);i;R_i}).$$ Furthermore,
\begin{align*}
\kappa_i^{-1}(w_{j}^{\kappa_i(\theta);i;R_i} w_{k}^{\kappa_i(\theta);i;R_i})
&= \kappa_i^{-1}(w_{j}^{\kappa_i(\theta);i;R_i})
\kappa_i^{-1}( w_{k}^{\kappa_i(\theta);i;R_i})\\
&=w_j^{\theta;i}(w_i^{\theta;i})^*
w_k^{\theta;i}(w_i^{\theta;i})^* \\
&=e^{2\pi i\bigl(-\theta_{ik} + \theta_{jk} + \theta_{ij}\bigr)}
w_k^{\theta;i}
(w_i^{\theta;i})^*
w_j^{\theta;i}
(w_i^{\theta;i})^*\\
&=e^{2\pi i\bigl( \theta_{ij} + \theta_{jk} + \theta_{ki} \bigr)}
\kappa_i^{-1}(w_{k}^{\kappa_i(\theta);i;R_i})
\kappa_i^{-1}(w_{j}^{\kappa_i(\theta);i;R_i})\\
&=e^{2\pi i \kappa_i(\theta)_{jk}}
\kappa_i^{-1}(w_{k}^{\kappa_i(\theta);i;R_i}
w_{j}^{\kappa_i(\theta);i;R_i}).
\end{align*}
\item Since $(w_{j}^{\kappa_i(\theta);i;R_i})^* w_{k}^{\kappa_i(\theta);i;R_i} =
    e^{2\pi i -\kappa_i(\theta)_{jk}} w_{k}^{\kappa_i(\theta);i;R_i}
    (w_{j}^{\kappa_i(\theta);i;R_i})^*$, we see that
$$\kappa_i^{-1}\bigl((w_{j}^{\kappa_i(\theta);i;R_i})^*
w_{k}^{\kappa_i(\theta);i;R_i}\bigr) = e^{-2\pi i \kappa_i(\theta)_{jk}}
\kappa_i^{-1}\bigl(w_{k}^{\kappa_i(\theta);i;R_i}
(w_{j}^{\kappa_i(\theta);i;R_i})^*\bigr).$$ Furthermore,
\begin{align*}
\kappa_i^{-1}\bigl((w_{j}^{\kappa_i(\theta);i;R_i})^*
w_{k}^{\kappa_i(\theta);i;R_i}\bigr)
& = \kappa_i^{-1}(w_{j}^{\kappa_i(\theta);i;R_i})^*
\kappa^{-1}_i(w_{k}^{\kappa_i(\theta);i;R_i})\\
&=
w_i^{\theta;i}
(w_j^{\theta;i})^*
w_k^{\theta;i}
(w_i^{\theta;i})^*\\
&=e^{2\pi i\bigl(-\theta_{jk} + \theta_{ik} + \theta_{ji} \bigr)}
w_k^{\theta;i}
(w_i^{\theta;i})^*
w_i^{\theta;i}
(w_j^{\theta;i})^*\\
&=e^{2\pi i\bigl(- \theta_{ij} -\theta_{jk} - \theta_{ki} \bigr)}
\kappa_i^{-1}(w_{k}^{\kappa_i(\theta);i;R_i})^*
\kappa^{-1}_i(w_{j}^{\kappa_i(\theta);i;R_i})\\
&=e^{-2\pi i \kappa_i(\theta)_{jk}} \kappa_i^{-1}\bigl(w_{k}^{\kappa_i(\theta);i;R_i}
(w_{j}^{\kappa_i(\theta);i;R_i})^*\bigr).
\end{align*}
\end{enumerate}
Thus we have shown that $\kappa_i$ and $\kappa_i^{-1}$ are well defined
$^*$-homomorphisms. They are evidently $U(1)$-equivariant. Since $w_i^{\theta;i}$ and
$w_{i}^{\kappa_i(\theta);i;R_i}$ are unitaries, $\kappa_i$ and $\kappa_i^{-1}$ are
mutually inverse.
\end{proof}

The maps $\kappa_i, \kappa^{-1}_i$ descend to the $B^\theta_{ij}$ because they fix the
generator
$$
\sigma_i\big(1 - w^\theta_j (w^\theta_j)^*\big)
$$
of $\sigma_i(I^\theta_j)$. It follows that $\kappa_i$ induces an invertible
$U(1)$-equivariant $C^*$-isomorphism\linebreak
\mbox{$\kappa_{i;j}:B^\theta_{ij}\rightarrow B^{\kappa_i(\theta);R_i}_{ij}$} such that
\begin{gather*}
\kappa_{i;j}(w_k^{\theta;ij}) = w_k^{\kappa_i(\theta);ij;R_i}
w_i^{\kappa_i(\theta);ij;R_i} \ \text{if\ }i\neq k,\quad \kappa_{i;j}(w_i^{\theta;ij}) =
w_i^{\kappa_i(\theta);ij;R_i},\\
\kappa_{i;j}^{-1}(w_k^{\kappa_i(\theta);ij;R_i}) = w_k^{\theta;ij}(w_i^{\theta;ij})^* \
\text{if\ }i\neq k,\quad \kappa_{i;j}^{-1}(w_i^{\kappa_i(\theta);ij;R_i}) =
w_i^{\theta;ij}.
\end{gather*}

So we obtain maps $\hat{\sigma}^i_j : B_i^{\kappa_i(\theta);R_i} \to
B_{ij}^{\kappa_i(\theta);R_i}$ from the commuting diagram
\begin{equation*}
\xymatrix{
B^{\kappa_i(\theta);R_i}_i\ar[r]^-{\kappa_i^{-1}}
\ar@/_2em/[rrr]_{\hat\sigma^i_j}
&B^\theta_i
\ar[r]^{\pi^i_j} & B^\theta_{ij}
\ar[r]^-{\kappa_{i;j}}& B^{\kappa_i(\theta);R_i}_{ij}.
}
\end{equation*}
For any $0\leq k\leq N$, we have $\hat\sigma^i_j(w^{\kappa_i(\theta);i;R_i}_k) =
w^{\kappa_i(\theta);ij;R_i}_k$.

\subsection{The multipullback structure of
\texorpdfstring{$C(S^{2N+1}_{H,\theta})^R$}{the gauged twisted sphere algebra}}

We define the twisted Heegaard sphere $C^*$-algebra $C(S^{2N+1}_{H,\theta})^R$ to be the
image of $C(S^{2N+1}_{H,\theta})$ under $\prod_{i=0}^N\kappa_i$. We compute morphisms
$\hat\pi^i_j$ that assemble the $B_i^{\kappa_i(\theta);R_i}$ into the multipullback
$C^*$-algebra $C(S^{2N+1}_{H,\theta})^R$. Fix any $i<j$. We determine $\hat\pi^i_j$ and
$\hat\pi^j_i$ through the commutative diagram
\begin{equation*}
\xymatrix{
B^{\kappa_i(\theta);R_i}_i\ar[d]_{\kappa_i^{-1}}
\ar@{-->}[rrd]^{\hat\pi^i_j}
&&&&
B^{\kappa_j(\theta);R_j}_j\ar[d]^{\kappa_j^{-1}}
\ar@{-->}[lld]_{\hat\pi^j_i}
\\
B^\theta_i \ar[r]_-{\pi^i_j} & B^\theta_{ij}\ar[r]_-{\kappa_{i;j}} & B^{\kappa_i(\theta);R_i}_{ij}
& B^\theta_{ij}\ar[l]^-{\kappa_{i;j}} & B^\theta_i\ar[l]^-{\pi^j_i}
}.
\end{equation*}
Then $C(S^{2N+1}_{H,\theta})^R$ is equivariantly isomorphic to the multipullback
$C^*$-algebra over the $\hat \pi^i_j$. Note that the above diagram can be rewritten as
follows:
\begin{equation}
\xymatrix{
B^{\kappa_i(\theta);R_i}_i \ar[dd]^{\hat\sigma^i_j}
\ar@/_1em/@{-->}[dd]_{\hat\pi^i_j}
&&
B^{\kappa_j(\theta);R_j}_j\ar[dd]^{\hat\sigma^j_i}
\ar@{-->}[lldd]_{\hat\pi^j_i}
\\
&&\\
B^{\kappa_i(\theta);R_i}_{ij}&
B^\theta_{ij}\ar[l]_-{\kappa_{i;j}}&
B^{\kappa_j(\theta);R_j}_{ij}\ar[l]_-{\kappa_{j;i}^{-1}}
\ar@/^1em/@{-->}[ll]^{\hat\psi_{ij}}
}.
\end{equation}
Thus, for $i<j$, we have $\hat\pi^i_j:=\hat\sigma^i_j$ and
$\hat\pi^j_i:=\hat\psi_{ij}\circ\hat\sigma^i_j$, where
$\hat\psi_{ij}:=\kappa_{i;j}\circ\kappa_{j;i}^{-1}$. We compute the images of the
generators of $B^{\kappa_j(\theta);R_j}_{ij}$ under the $\hat\psi_{ij}$: for $i<j$ and
$k\neq i,j$,
\begin{subequations}
\label{HatPsiDef}
\begin{align*}
\hat\psi_{ij}(w^{\kappa_j(\theta);ij;R_j}_k) &:=
\kappa_{i;j}\bigl(\kappa_{j;i}^{-1}(w^{\kappa_j(\theta);ij;R_j}_k)\bigr)\\
&=\kappa_{i;j}\bigl(w^{\theta;ij}_k(w^{\theta;ij}_j)^*\bigr)\\
&=\kappa_{i;j}(w^{\theta;ij}_k)\kappa_{i;j}(w^{\theta;ij}_j)^*\\
&=w_k^{\kappa_i(\theta);ij;R_i}w_i^{\kappa_i(\theta);ij;R_i}
(w_i^{\kappa_i(\theta);ij;R_i})^*
(w_j^{\kappa_i(\theta);ij;R_i})^*\\
&=w_k^{\kappa_i(\theta);ij;R_i}
(w_j^{\kappa_i(\theta);ij;R_i})^*,\\
\hat\psi_{ij}(w^{\kappa_j(\theta);ij;R_j}_i) &:=
\kappa_{i;j}\bigl(\kappa_{j;i}^{-1}(w^{\kappa_j(\theta);ij;R_j}_i)\bigr)\\
&=\kappa_{i;j}\bigl(w^{\theta;ij}_i(w^{\theta;ij}_j)^*\bigr)\\
&=w_i^{\kappa_i(\theta);ij;R_i}
(w_i^{\kappa_i(\theta);ij;R_i})^*
(w_j^{\kappa_i(\theta);ij;R_i})^*\\
&=(w_j^{\kappa_i(\theta);ij;R_i})^*,\quad\text{ and}\\
\hat\psi_{ij}(w^{\kappa_j(\theta);ij;R_j}_j) &:=
\kappa_{i;j}\bigl(\kappa_{j;i}^{-1}(w^{\kappa_j(\theta);ij;R_j}_j)\bigr)\\
&=\kappa_{i;j}(w^{\theta;ij}_j)\\
&=w_j^{\kappa_i(\theta);ij;R_i}w_i^{\kappa_i(\theta);ij;R_i}.
\end{align*}
\end{subequations}

\subsection{The \texorpdfstring{$U(1)$}{U(1)}-fixed-point subalgebra of \texorpdfstring{$C(S^{2N+1}_{H,\theta})^{R}$}{of the twisted sphere algebra} as a multipullback}

For any antisymmetric $(N+1)\times (N+1)$ real matrix $\theta$, let us denote by
$\check\kappa_i(\theta)$ the matrix obtained from $\kappa_i(\theta)$ by removing the
$i$-th row and column. Re-index the remaining elements so that both row and column indices
run from $1$ to $N$.

For any  $0\leq i\leq N$, let $A_i:=\Tt_{\check\kappa_i(\theta)}^N$. The isometries
$v_1^i,\ldots,v_N^i$ generating $A_i$ satisfy
\begin{equation*}
v^i_jv^i_k=e^{2\pi i\check\kappa_i(\theta)_{jk}}v^i_k v^i_j,\quad
(v^i_j)^*v^i_k=e^{-2\pi i\check\kappa_i(\theta)_{jk}}v^i_k (v^i_j)^*,
\end{equation*}
for all $1\leq j,k\leq N$, $j\neq k$.

We claim that $A_i$ is isomorphic as a $C^*$-algebra with the $U(1)$-invariant subalgebra
of $B^{\kappa_i(\theta);R_i}_i$. To see this, observe that the universal property of
$A_i$ yields a $C^*$-homomorphism $\phi_i : A_i \to B^{\kappa_i(\theta);R_i}_i$ such that
\begin{equation*}
\phi_{i}(v^i_k)=\begin{cases}
w^{\kappa_i(\theta);i;R_i}_{k-1}& \text{if}\ k\leq i\\
w^{\kappa_i(\theta);i;R_i}_{k} & \text{if}\ k > i.
\end{cases}
\end{equation*}
An argument using the gauge-invariant uniqueness theorem as in the proof of
Theorem~\ref{prp:pullback isomorphism} shows that $\phi_i$ is injective. To see that it
is surjective, first observe that $B^{\kappa_i(\theta);R_i}_i$ is densely spanned by
elements of the form
\begin{equation*}
(w^{\kappa_i(\theta);i;R_i}_{1})^{n_1} \cdots (w^{\kappa_i(\theta);i;R_i}_{N})^{n_N}
(w^{\kappa_i(\theta);i;R_i}_{N})^{*m_N}
\cdots (w^{\kappa_i(\theta);i;R_i}_{1})^{*m_1}.
\end{equation*}
 Since $w^{\kappa_i(\theta);i;R_i}_{i}$ is unitary in $B^{\kappa_i(\theta);R_i}$,
the expectation onto the $U(1)$-invariant subalgebra of
$B^{\kappa_i(\theta);R_i}_i$, obtained by averaging over the $U(1)$-action, takes such a
spanning element to
\begin{equation*}
 \delta_{n_i, m_i} \prod_{\substack{j=0\\ j \not= i}}^N(w^{\kappa_i(\theta);i;R_i}_{j})^{n_j} \prod_{\substack{k=N\\ k \not=i}}^0
(w^{\kappa_i(\theta);i;R_i}_{k})^{*n_k}.
\end{equation*}
Therefore, the $U(1)$-invariant subalgebra of $B^{\kappa_i(\theta);R_i}_i$ is spanned by
elements of this form, and such elements are in the range of $\phi_i$. Hence $\phi_i$ is
surjective. For any $i\neq j$ we will denote the generators of $A_{i;j}$ (which are the
images under the canonical quotient maps of the generators of $A_i$) by $v_1^{i;j},\ldots
v_N^{i;j}$. For $i<j$, the elements $v_j^{i;j}\in A_{i;j}$ and $v_{i+1}^{j;i}\in A_{j;i}$
are unitary. The inverse of $\phi_i$ satisfies
\begin{equation*}
\phi_i^{-1}(w_k^{\kappa_i(\theta);i;R_i})
    = \begin{cases}
        v^i_{k+1} &\text{if\ }k<i\\
        v^i_k &\text{if\ }i<k.
    \end{cases}
\end{equation*}

Let $J_j$ be the ideal of $A_i$ generated by $(1-v^i_j(v^i_j)^*)$. For $0\leq i< j\leq
N$, let
\begin{equation*}
A_{i;j}:=A_i/J_j,\quad\text{ and }\quad A_{j;i}:=A_j/J_{i+1}.
\end{equation*}

The isomorphisms $\phi_i^{-1}$ descend to isomorphisms
\begin{equation*}
\phi_{ij}^{-1} : (B^{\kappa_i(\theta);R_i}_{ij})^{U(1)}\rightarrow A_{i;j},
\quad
w_k^{\kappa_i(\theta);ij;R_i}\mapsto \begin{cases}
v^{i;j}_{k+1} &\text{if\ }k<i\\
v^{i;j}_k &\text{if\ }i<k.
\end{cases}
\end{equation*}

Using the isomorphisms $\phi_i$ and $\phi_{ij}$ we can transport the multipullback
structure of the $U(1)$-fixed-point subalgebra of $C(S^{2N+1}_{H,\theta})^{R}$ as follows ($0\leq i<j\leq N$):
\begin{equation}\label{eq:transported multipullback}
\xymatrix{
A_i \ar@{-->}@/_1em/[ddr]_{\rho^i_j}
\ar[r]^-{\phi_i}& \bigl(B_i^{\kappa_i(\theta);R_i}\bigr)^{U(1)}
\ar[d]_{\hat\sigma^i_j}
&& \bigl(B_j^{\kappa_j(\theta);R_j}\bigr)^{U(1)}
\ar[d]^{\hat\sigma^j_i}
& A_j\ar[l]_-{\phi_j}
\ar@{-->}@/^1em/[ddl]^{\rho^j_i}
\\
& \bigl(B_{ij}^{\kappa_i(\theta);R_i}\bigr)^{U(1)}
\ar[d]_{\phi_{ij}^{-1}}
 &&
\bigl(B_{ij}^{\kappa_j(\theta);R_j}\bigr)^{U(1)}
\ar[d]^{\phi_{ji}^{-1}}
\ar[ll]_-{\hat\psi_{ij}} &\\
& A_{i;j} && A_{j;i}\ar@{-->}[ll]^-{\psi_{ij}}
}.
\end{equation}
In the diagram~\eqref{eq:transported multipullback}, we have used the same symbols to
denote the (co)-restrictions of the  maps $\hat\sigma^i_j$, $\hat\sigma^j_i$
and $\hat\psi_{ij}$ to the respective $U(1)$-invariant subalgebras. Since all these maps are
$U(1)$-equivariant, the restrictions corestrict as expected.

We will now explicitly write the values of maps $\rho^i_j$, $\rho^j_i$, $\psi_{ij}$,
$0\leq i<j\leq N$, defined by the commutative diagram above, on generators of respective
domains. It is straightforward to verify that $\rho^i_j$ and $\rho^j_i$ are the canonical
quotient maps given by
 \begin{equation*}
 \rho^i_j(v_k^i)=v^{i;j}_k,\quad\text{ and }\quad
 \rho^j_i(v_k^j)=v^{j;i}_k,\quad 1\leq k\leq N.
 \end{equation*}
 In case of the isomorphisms
 $\psi_{ij}:=\phi_{ij}^{-1}\circ\hat\psi_{ij}\circ\phi_{ji}$,
  $0\leq i<j\leq N$,
  we will perform a careful case-by-case analysis. The first splitting into cases
  follows from the definition of $\hat\psi_{ij}$ (see~\eqref{HatPsiDef}):
 either $k=i+1$
  or $k\neq i+1$.
  \begin{enumerate}
  \item For $k=i+1$:
  \begin{align*}
  \psi_{ij}(v_{i+1}^{j;i})&=\phi_{ij}^{-1}\Bigl(\hat\psi_{ij}\bigl(\phi_{ji}(v_{i+1}^{j;i}
  )\bigr)\Bigr)\\
  &=\phi_{ij}^{-1}\Bigl(\hat\psi_{ij}\bigl(w_i^{\kappa_j(\theta);ij;R_j}\bigr)\Bigr)\\
  &=\phi_{ij}^{-1}\Bigl(\bigl(w_j^{\kappa_i(\theta);ij;R_i}\bigr)^*\Bigr)\\
  &=(v^{i;j}_j)^*.
  \end{align*}
  \item For $k\neq i+1$:
  $
  \psi_{ij}(v_{k}^{j;i})=\phi_{ij}^{-1}\Bigl(\hat\psi_{ij}\bigl(\phi_{ji}(v_{k}^{j;i}
  )\bigr)\Bigr)=:(*)$.
  Here the definition of $\phi_{ji}$ forces a split into cases  $k>j$ or $k\leq j$.
  \begin{enumerate}
  \item For $k>j$:
  \begin{align*}
  (*)&=\phi_{ij}^{-1}\Bigl(\hat\psi_{ij}\bigl(w_k^{\kappa_j(\theta);ij;R_j}\bigr)\Bigr)\\
  &=\phi_{ij}^{-1}\Bigl(
  w_k^{\kappa_i(\theta);ij;R_i}\bigl(w_j^{\kappa_i(\theta);ij;R_i}\bigr)^*
  \Bigr)\\
  &=v_k^{i;j}(v_j^{i;j})^*.
  \end{align*}
  \item For $k\leq j$:
  \begin{align*}
  (*)&=\phi_{ij}^{-1}\Bigl(\hat\psi_{ij}\bigl(w_{k-1}^{\kappa_j(\theta);ij;R_j}\bigr)\Bigr)
  \\
  &=\phi_{ij}^{-1}\Bigl(
  w_{k-1}^{\kappa_i(\theta);ij;R_i}\bigl(w_j^{\kappa_i(\theta);ij;R_i}\bigr)^*
  \Bigr)\\
  &=:(**).
  \end{align*}
  Now we arrive at another split into cases:  $k-1>i$ or $k-1<i$. (The case
  $k-1=i$ was taken care of previously.)
  \begin{enumerate}
  \item If $k-1>i$, then $(**)=v_{k-1}^{i;j}(v_j^{i;j})^*$.
  \item If $k-1<i$, then $(**)=v_{k}^{i;j}(v_j^{i;j})^*$.
  \end{enumerate}
  \end{enumerate}
  \end{enumerate}

Summarizing, when $0\leq i<j\leq N$ and $1\leq k\leq N$,  we obtain
\begin{equation*}
\psi_{ij}(v_k^{j;i})=\begin{cases}
v_j^{i;j} &\text{if}\ k=i+1\\
v_k^{i;j}(v_j^{i;j})^*&\text{if}\ k>j \text{ or } k<i+1\\
v_{k-1}^{i;j}(v_j^{i;j})^*&\text{if}\ i+1<k\leq j
\end{cases}.
\end{equation*}
Consequently, the $U(1)$-fixed-point subalgebra of $C(S^{2N+1}_{H,\theta})^R$ is isomorphic to the multipullback of the algebras $A_i$
with respect to the natural maps $A_i\to A_{i;j}$, $A_j \to A_{i;j}$, $i < j$, determined by the
diagrams
\begin{equation*}
\xymatrix{
A_i\ar[d]_{\rho^i_j} & A_j\ar[d]^{\rho^j_i}\\
A_{i;j} & A_{j;i}.\ar[l]^{\psi_{ij}}
}
\end{equation*}

\section{The \texorpdfstring{$K$}{K}-groups of twisted multipullback quantum odd spheres and complex projective
spaces}\label{sec:PSpace Kth}


\noindent
We begin by deriving a short exact sequence of commutative $C^*$-algebras whose noncommutative counterpart
provides a basis for computing the $K$-groups of the twisted multipullback quantum complex projective spaces.

The $2N+1$-dimensional sphere $S^{2N+1}$ is the closed subset of $\mathbb{C}^{N+1}$
defined by
 \begin{equation*}
 S^{2N+1}=\Big\{(z_0,\ldots,z_N)\in \mathbb{C}^{N+1}\;\Big|\;
 \textstyle{\sum_{i=0}^N}|z_i|^2=1\Big\}.
 \end{equation*}
Denote by $D:=\{c \in\mathbb{C}\;|\;|c|\leq 1\}$ the unit disk, and by  $D_0:=\{c
\in\mathbb{C}\;|\;|c|< 1\}$ the interior of the unit disk. Next, we define a ``non-round" odd sphere as follows:
 \begin{equation*}
 S_D^{2N+1}:=\Big\{(c_0,\ldots,c_N)\in D^{N+1}\;\Big|\;\textstyle{\prod_{i=0}^N}(1-|c_i|^2)=0\Big\}.
 \end{equation*}
Since $\prod_{i=0}^N(1-|c_i|^2)=0$ if and only if $|c_i|=1$ for some
$i\in\{0,\ldots,N\}$, it follows that $\sum_{i=0}^N|c_i|^2\geq 1$ for any
$(c_0,\ldots,c_N)\in S_D^{2N+1}$. Also, $\sum_{i=0}^N|z_i|^2=1$ gives
$$
\max\{|z_0|,\ldots,|z_N|\}\geq\frac{1}{\sqrt{N+1}}.
$$
 Hence there are well-defined maps
\begin{gather*}
 S_D^{2N+1}\ni (c_j)_{j=0}^N\longmapsto \left(
 \frac{c_j}{\sqrt{\sum_{i=0}^N|c_i|^2}}\right)^N_{j=0}\in S^{2N+1},\\
 S^{2N+1}\ni (z_j)^N_{j=0}\longmapsto \left(
 \frac{z_j}{\max\{|z_0|,\ldots,|z_N|\}}
 \right)^N_{j=0} \in  S_D^{2N+1}.
 \end{gather*}
These maps are mutually inverse and continuous, so that $S^{2N+1} \cong S_D^{2N+1}$.

Now consider the following splitting of $S_D^{2N+1}$ into a pair
of disjoint sets which are closed and open respectively:
\begin{equation*}
S_D^{2N+1}=\{(c_i)_i\in S_D^{2N+1}\;|\; |c_N|=1\}\;\textstyle{\coprod}\;
\{(c_i)_i\in S_D^{2N+1}\;|\; |c_N|<1\}.
\end{equation*}
The condition in the first of these sets forces $\prod_{i=0}^N(1-|c_i|^2)=0$ regardless
of the values of $(c_0,\ldots,c_{N-1})\in D^N$. Hence
 \begin{equation*}
 \{(c_i)_i\in S_D^{2N+1}\;|\; |c_N|=1\}=D^N\times S^1.
 \end{equation*}
 Furthermore, when $(c_i)_{i=0}^N$ is an element  of the second set, then
 $\prod_{i=0}^{N-1}(1-|c_i|^2)=0$ because $1-|c_N|^2>0$. Consequently,
 \begin{equation*}
 \{(c_i)_i\in S_D^{2N+1}\;|\; |c_N|<1\}=S_D^{2N-1}\times D_0.
 \end{equation*}
 Summarizing, we obtain the decomposition
 \begin{equation*}
 S_D^{2N+1}=\big(D^N\times S^1\big)\;\textstyle{\coprod}\;\big(S_D^{2N-1}\times D_0\big).
 \end{equation*}
For the diagonal actions of~$U(1)$, this decomposition of $S_D^{2N+1}$ induces the $U(1)$-equivariant short exact sequence
\begin{equation*}
\xymatrix{
0\ar[r]&C_0(S_D^{2N-1}\times D_0)\ar[r]&
C(S_D^{2N+1})\ar[r]&C(D^N\times S^1)\ar[r]
&0}
\end{equation*}
of $C^*$-algebras.
Finally, remembering that $S_D^{2N-1}$ and $S^{2N-1}$ are equivariantly
homeomorphic for the diagonal $U(1)$-actions, and using standard identifications,
 we obtain the following $U(1)$-equivariant short exact sequence of $C^*$-algebras:
 \begin{equation}
 \label{short-classical}
\xymatrix{
0\ar[r]&C(S^{2N-1})\otimes C_0(D_0)\ar[r]&
C(S^{2N+1})\ar[r]&C(D)^{\otimes N}\otimes C(S^1)\ar[r]
&0
}.
\end{equation}

\subsection{Quantum odd spheres}

Recall that $s$ denotes the isometry generating the Toeplitz algebra $\Tt$. The universal
properties of the maximal tensor product and of the untwisted algebra $\Tt^{N+1}_0$ show
that the map
\begin{equation} \label{eq:Tiso}
\Tt^{N+1}_0 \ni w_{j} \longmapsto 1^{\otimes j} \otimes s \otimes 1^{\otimes N-j} \in \Tt^{\otimes N+1}
\end{equation}
is an isomorphism.

To see where Definition~\ref{dfn:twisted sphere} comes from, and how it relates to
noncommutative solid tori, recall first that $\sigma$ denotes the symbol map from $\Tt$ to
$C(\mathbb{T})$. When $\theta = 0$, we denote $C(S^{2N+1}_{H, \theta})$ by
$C(S^{2N+1}_H)$. We have $\Tt_0^{N+1} = \Tt^{\otimes N+1}$, and
 each $I_i$ of Definition~\ref{dfn:twisted sphere} is precisely the kernel of
\begin{equation}\label{sigmai}
\id^{\otimes i} \otimes \sigma \otimes
\id^{\otimes N-i} :  \Tt^{\otimes N+1} \longrightarrow
B_i:=\Tt^{\otimes i}\otimes C(\mathbb{T})\otimes \Tt^{\otimes {N-i}} ,
\end{equation}
and so each $B_i$ is the noncommutative solid torus algebra $\Tt^{\otimes i} \otimes
C(\mathbb{T}) \otimes \Tt^{\otimes N-i}$. The algebras $B_i$ and $B_{ij}$ and the maps
$\pi^i_j$ of Definition~\ref{dfn:twisted sphere} are then given by
\begin{gather}
 B_{ij}:=\Tt^{\otimes i}\otimes C(\mathbb{T})\otimes\Tt^{\otimes{j-i-1}}\otimes C(\mathbb{T})
\otimes\Tt^{\otimes{N-j}},\quad
i<j,\quad i,j\in\{0,1,\ldots,N \},\nonumber\\
 B_{ij}:=B_{ji},\quad j<i,\quad i,j\in\{0,1,\ldots,N \},\quad\text{ and} \nonumber\\
\label{piij}
\pi^i_j := \id^j\otimes\sigma\otimes\id^{N-j} : B_i\rightarrow B_{ij},\quad i\neq j,\quad
i,j\in\{0,1,\ldots,N\}.
\end{gather}
\noindent
Thus our definition of $C(S^{2N+1}_{H, \theta})$ as the multipullback along the $\pi^i_j$
is a natural noncommutative dual to the Heegaard-type splitting of $S^{2N+1}$ described
in Section~\ref{sec:Hsplittings}.

To compute $K_*(C(S^{2N+1}_{H, \theta}))$, we first compute the $K$-theory of the
untwisted quantum sphere $C(S^{2N+1}_H)$ by applying the K\"unneth theorem and then the
six-term ideal-quotient exact sequence. We then apply results of
\cite{SimsWhiteheadEtAl:xx13} to see that the $K$-theory of $C(S^{2N+1}_{H,\theta})$ is
identical to that of~$C(S^{2N+1}_H)$. Since the cocycle $c$ on $\Lambda$ in
Lemma~\ref{lem:k-graph alg} is induced by a group cocycle on $\ZZ^k$, the corresponding
twisted multiplication on $C^*(\Lambda; \mathcal{E})$ can be realised using Rieffel's
framework of twisted multiplicative structures on $C^*$-algebras arising from actions of
$\RR^k$ applied to the gauge action of $\TT^k$ on $C^*(\Lambda; \Ee)$ and the dense
$^*$-subalgebra $\operatorname{span}\{s_\mu s^*_\nu : \mu,\nu \in \Lambda\}$. So we could
alternatively apply \cite[Main Theorem (page~200)]{Rieffel} to prove that the $K$-theory
of $C(S^{2N+1}_{H,\theta})$ is identical to that of~$C(S^{2N+1}_H)$.

Recall that $\Tt^{N+1}_0$ is canonically isomorphic to $\Tt^{\otimes N+1}$ via the map
that carries the generator $w_i$ of $\Tt^{N+1}_0$ to the elementary tensor $1 \otimes
\cdots \otimes 1 \otimes s \otimes 1 \otimes \cdots \otimes 1$, where the $s$ appears in
the $i$th (counting from zero) tensor factor. Recall also that we have $K_0(\Tt) = \ZZ$
and $K_1(\Tt) = 0$ with the generator in $K_0$ being the class of the identity element.
It then follows from the K\"unneth theorem (see, e.g., \cite[Remarks~9.3.3]{Wegge-Olsen})
that $K_0(\Tt^{N+1}_0) = \ZZ[1]$ and $K_1(\Tt^{N+1}_0) = 0$. Given $m = (m_0, m_1, \dots,
m_N) \in \ZZ^{N+1}$, we write $W_m$ for the element $\prod^N_{i=0} w_i^{m_i}$ of
$\Tt^{N+1}_0$. (By convention, $w_i^{-k} = (w_i^*)^k$ for $k \ge 0$.)

\begin{lem} \label{lem:inclusion}
For $N \ge 0$, there is an isomorphism of $\Kk(\ell^2(\NN^{N+1}))$ onto the ideal $I$ of
$\Tt^{N+1}_0$ generated by $\prod^N_{j=0} (1 - w_j w^*_j)$ that carries the matrix unit
$E_{pq}$ to
\begin{equation*}
W_p \Big(\prod^N_{j=0} (1 - w_j w^*_j)\Big) W_q^* .
\end{equation*}
\end{lem}

\begin{proof}
Let $ R := \prod^N_{j=0} (1 - w_j w^*_j) $. As the $w_i$ are commuting isometries, we see
that $w_i^* R = 0 = R w_i$ for all $i$, and then we deduce that $W^*_p R = 0 = R W_p$ for
all $p \in \NN^{N+1}\setminus\{0\}$. Similarly, observe that
\begin{equation*}
(W_p R W^*_q)(W_a R W^*_b) = (W_p R) W_q^* W_a R W_b^* = \delta_{q,a} w_p R w^*_b.
\end{equation*}
Since $(W_p R W^*_q)^* = W_q R W^*_p$, we see that the $W_p R W^*_q$ form a family of
matrix units indexed by $\NN^{N+1}$, and so there is a homomorphism of
$\Kk(\ell^2(\NN^{N+1})) \to I$ carrying each $E_{pq}$ to $W_p R W^*_q$. Since $R$ is
nonzero, and since $\Kk(\ell^2(\NN^{N+1}))$ is simple, this homomorphism is injective.
Surjectivity follows from
\begin{equation*}
\textstyle\prod^N_{j=0} (1 - w_j w^*_j) =  (1 - w_0 w^*_0 ) R = R - w_0 R w_0^*. \qedhere
\end{equation*}
\end{proof}

The following result generalizes \cite[Theorem~4.1]{bhms05} and \cite[Theorem~3.2]{hr}.
It also contains statement~\eqref{it:main-sphere-Kth} of Theorem~\ref{thm:main}, so the
proof of this theorem also proves Theorem~\ref{thm:main}\eqref{it:main-sphere-Kth}.
\begin{thm}
Consider an integer $N \ge 1$ and an antisymmetric matrix $\theta \in M_{N+1}(\RR)$. Then
$K_1(C(S^{2N+1}_{H,\theta})) \cong \ZZ$ and there is an isomorphism
$K_0(C(S^{2N+1}_{H,\theta}))$ $\cong \ZZ$ that carries $[1_{C(S^{2N+1}_{H,\theta})}]$ to~$1$.
\end{thm}
\begin{proof}
We first consider the case where $\theta_{ij} = 0$ for all $i,j$.
Theorem~\ref{prp:pullback isomorphism} combined with Lemma~\ref{lem:inclusion} and the
isomorphism $\Tt^{N+1}_0 \cong \Tt^{\otimes N+1}$ given in~\eqref{eq:Tiso} implies that
\begin{equation}
\label{TbyKEq}
C(S^{2N+1}_H) \; \cong \; \Tt^{N+1}_0/I \; \cong \; \Tt^{\otimes N+1} / \Kk(\ell^2(\NN^{N+1})) .
\end{equation}

We claim that the inclusion $\iota : \Kk(\ell^2(\NN^{N+1})) \to \Tt^{N+1}_0$ of
Lemma~\ref{lem:inclusion} induces the zero map on $K$-theory. As
$K_0(\Kk(\ell^2(\NN^{N+1}))) \cong \ZZ$ is generated by $[R]$, we just have to show that
$[R] = 0$ in $K_0 ( \Tt^{N+1}_0 ) \cong \ZZ$. The isomorphism $\Tt^{N+1}_0 \cong
\Tt^{\otimes N+1}$ given by~\eqref{eq:Tiso} carries $R$ to $(1 - ss^*) \otimes (1 - ss^*)
\otimes \cdots \otimes (1 - ss^*)$. Since $s$ is an isometry, we have $[1 - ss^*] = [s^*s
- ss^*] = 0$ in $K_0(\Tt)$. As $K_1 ( \Tt ) = 0$, the K\"unneth isomorphism implies that
$[(1 - ss^*) \otimes (1 - ss^*) \otimes \cdots \otimes (1 - ss^*)] = 0$ in
$K_0(\Tt^{\otimes N+1})$. Therefore $[R]$ is zero in $K_0(\Tt^{N+1}_0)$ as claimed.

Since $K_1( \Kk(\ell^2(\NN^{N+1})) ) = 0 = K_1 (\Tt^{N+1}_0)$, Theorem~9.3.2 of
\cite{Wegge-Olsen} gives an exact sequence
\begin{equation*}
\begin{tikzpicture}
	 \node (k0j) at (0,2) {$\ZZ$};
	 \node (k0t) at (3,2) {$\ZZ$};
	 \node (k0s) at (6,2) {$K_0(C(S^{2N+1}_H))$};
	 \node (k1j) at (6,0) {$\phantom{.}0$.};
	 \node (k1t) at (3,0) {$0$};
	 \node (k1s) at (0,0) {$K_0(C(S^{2N+1}_H))$};
	 \draw[-stealth] (k0j) to node[pos=0.5, above] {\small$0$} (k0t);
	 \draw[-stealth] (k0t) to (k0s);
	 \draw[-stealth] (k0s) to (k1j);
	 \draw[-stealth] (k1j) to (k1t);
	 \draw[-stealth] (k1t) to (k1s);
	 \draw[-stealth] (k1s) to (k0j);
\end{tikzpicture}
\end{equation*}
Hence $K_0(C(S^{2N+1}_H)) = \ZZ[1]$ and $K_1(C(S^{2N+1}_H)) \cong \ZZ$.

For general $\theta$, we have $C(S^{2N+1}_{H,\theta}) \cong C^*(\Lambda, c; \Ee)$ by
Lemma~\ref{lem:k-graph alg}. By~\eqref{eq:cdef}, the cocycle $c$ on $\Lambda$ arises from
exponentiation of an $\RR$-valued cocycle. Hence
\cite[Theorem~6.1]{SimsWhiteheadEtAl:xx13} gives
\begin{equation*}
K_*( C(S^{2N+1}_{H,\theta}) ) \; \cong  \;  K_*( C^*(\Lambda, c; \Ee) ) \; \cong \;
  K_*( C^*(\Lambda, 1 ; \Ee) ) \; \cong \; K_* ( C(S^{2N+1}_{H}) )
\end{equation*}
via isomorphisms that preserve the $K_0$-class of the identity.
\end{proof}
\begin{rmk}An alternative proof can be obtained using the exact sequence~\eqref{short-exact-main}.
\end{rmk}

\subsection{Multipullback quantum complex projective spaces}\label{untwisted}

In our computation of the $K$-theory of $C(\PTt)$, we will use two auxiliary results. The
first result is a quantum version of the short exact sequence~\eqref{short-classical}:
\begin{lem}
With respect to the diagonal $U(1)$-action, for any positive integer $k$, there exists a
$U(1)$-equivariant short exact sequence of $C^*$-algebras
\begin{equation}
\label{short-exact-main}
\xymatrix{
0\ar[r]&C(S^{2k-1}_H)\otimes \Kk\ar[r]&C(S^{2k+1}_H)
\ar[r]& \Tt^{\otimes k}\otimes C(S^1)\ar[r]& 0.
}
\end{equation}
\end{lem}
\begin{proof}
The starting point is the Toeplitz extension, i.e., the exact sequence
\begin{equation*}
\xymatrix{
0\ar[r]&\Kk\ar[r]&\Tt\ar[r]^-\sigma&C(S^1)\ar[r]&0,
}
\end{equation*}
where $\sigma$ is the symbol map. Since the Toeplitz algebra is nuclear, so is $\Tt^{\otimes k}$, whence
the sequence of $C^*$-algebras
\begin{equation}
\label{seq-then-lem}
\xymatrix{
0\ar[r]&\Tt^{\otimes k}\otimes\Kk\ar[r]&
\Tt^{\otimes k}\otimes\Tt\ar[r]^-{\id\otimes\sigma}&\Tt^{\otimes k}\otimes C(S^1)\ar[r]&0
}
\end{equation}
is also exact. Equation~\eqref{TbyKEq} gives $\big(\Tt^{\otimes k}\otimes
\Kk\big)/\Kk^{\otimes k+1}\cong C(S^{2k-1}_H)\otimes \Kk$ by the nuclearity of~$\Kk$. So taking quotients by
$\Kk^{\otimes k+1}$ throughout~\eqref{seq-then-lem} yields the exact
sequence~\eqref{short-exact-main}. The $U(1)$-equivariance follows from the fact that all
the identifications used are $U(1)$-equivariant.
\end{proof}

The second result is a standard fact about compact-group actions, so we omit its proof.
\begin{lem}
\label{invariant-exact-lem} Let $G$ be a compact Hausdorff topological group and let $A$ be a $C^*$-algebra with a pointwise norm
continuous $G$-action
$\alpha:G\rightarrow\mathrm{Aut}(A)$. Let $I\subseteq A$ be a
closed two-sided $G$-invariant ideal of~$A$. Then $A/I$ admits the induced $G$-action, and the sequence of fixed-point algebras
\begin{equation*}
\xymatrix{
0\ar[r] & I^G\ar[r]& A^G\ar[r]&(A/I)^G\ar[r]&0
}
\end{equation*}
is exact.
\end{lem}

To compute the $K$-groups of the invariant subalgebra $C(\PTt):=C(S_H^{2N+1})^\alpha$, we
first construct a family of short exact sequences. Fix $N\in\mathbb{N}$, $N\geq 1$. For
all $k\in \{1,\ldots,N\}$ apply the exact functor $\_\otimes\Kk^{\otimes N-k}$ to the
sequence~\eqref{short-exact-main} to obtain the short exact sequence
\begin{multline*}
\xymatrix{ 0\ar[r]&C(S^{2k-1}_H)\otimes\Kk^{\otimes N-k+1}\ar[r]&
C(S^{2k+1}_H)\otimes\Kk^{\otimes N-k}
}\\
\xymatrix{
 \ar[r]& \Tt^{\otimes k}\otimes C(S^1)\otimes\Kk^{\otimes N-k}\ar[r]& 0.
}
\end{multline*}
By Lemma~\ref{invariant-exact-lem}, the restriction of the above sequence to
$U(1)$-invariant subalgebras is again exact:
\begin{multline}
\label{short-exact-invariant-main} \xymatrix{
0\ar[r]&\Bigl(C(S^{2k-1}_H)\otimes\Kk^{\otimes N-k+1}\Bigr)^{U(1)}\ar[r]&
\Bigl(C(S^{2k+1}_H)\otimes\Kk^{\otimes N-k}\Bigr)^{U(1)}
}\\
\xymatrix{
 \ar[r]& \Bigl(\Tt^{\otimes k}\otimes C(S^1)\otimes\Kk^{\otimes N-k}\Bigr)^{U(1)}\ar[r]& 0
}.
\end{multline}
Our gauge trick \eqref{kappa}--\eqref{kappa-1} shows that $\Tt^{\otimes k}\otimes C(S^1)\otimes\Kk^{\otimes N-k}$ with
diagonal $U(1)$-action is $U(1)$-equivariantly isomorphic with $\Tt^{\otimes k}\otimes
C(S^1)\otimes \Kk^{\otimes N-k}$ where $U(1)$ acts only on the $C(S^1)$-component. Hence
\begin{equation}
(\Tt^{\otimes k}\otimes C(S^1)\otimes\Kk^{\otimes N-k})^{U(1)}\cong
\Tt^{\otimes k}\otimes\Kk^{\otimes N-k}.
\label{inv-obs-1}
\end{equation}

Next, let
\begin{equation*}
S_k:=\Bigl(C(S^{2k+1}_H)\otimes\Kk^{\otimes N-k}\Bigr)^{U(1)},\quad
k\in\{0,\ldots,N\}.
\end{equation*}
Using this notation and~\eqref{inv-obs-1}, we can write the family of
short exact sequences~\eqref{short-exact-invariant-main} as
\begin{equation}
\label{short-main-H}
\xymatrix{
0\ar[r]&S_{k-1}\ar[r]&S_k\ar[r]&
\Tt^{\otimes k}\otimes\Kk^{\otimes N-k}\ar[r]&0,
}
\end{equation}
where $k\in\{1,\ldots,N\}$.

\begin{thm}
Let $N$ be a positive integer. Then
$$
K_0(C(\PTt))=\mathbb{Z}^{N+1}\qquad \text{and}\qquad K_1(C(\PTt))=0.
$$
\end{thm}
\begin{proof}
Since $S_N = C(\PTt)$, it suffices to prove that $K_0(S_k) = \ZZ^{k+1}$ and $K_1(S_k) =
\{0\}$ for all $k\in\{1,\ldots,N\}$. We do this by induction on $k$. For $k = 0$, the
gauge trick gives
$$
S_0=\bigl(C(S^1)\otimes\Kk^{\otimes N}\bigr)^{U(1)}\cong\Kk^{\otimes N}.
$$
Consequently,
\begin{equation*}
K_0(S_0)\cong K_0(\Kk)=\mathbb{Z},\quad
K_1(S_0)\cong K_1(\Kk)=0.
\end{equation*}

Now assume that $K_0(S_{k-1}) = \mathbb{Z}^k$ and $K_1(S_{k-1}) = 0$. The short exact
sequence \eqref{short-main-H} of $C^*$-algebras induces the six-term exact sequence of
Abelian groups:
\begin{equation}
\label{six-term-complex}
\xymatrix{
K_0(S_{k-1})\ar[r]&K_0(S_k)\ar[r]&
K_0(\Tt^{\otimes k})\ar[d]\\
K_1(\Tt^{\otimes k})\ar[u]
&K_1(S_k)\ar[l]
&K_1(S_{k-1}).\ar[l]
}
\end{equation}
The K\"unneth theorem gives $K_0(\Tt^{\otimes k}) = \ZZ$ and $K_1(\Tt^{\otimes k}) =0$.
Combining this with the inductive hypothesis, the
sequence~\eqref{six-term-complex} becomes
\begin{equation*}
\xymatrix{
\mathbb{Z}^k \ar[r]&K_0(S_k)\ar[r]&
\mathbb{Z}\ar[d]\\
0\ar[u]
&K_1(S_k)\ar[l]
&\hspace*{1.5mm}0.\ar[l]
}
\end{equation*}
Exactness gives $K_1(S_k)=0$, and exactness combined with the projectivity of free
abelian groups gives $K_0(S_k)=\mathbb{Z}\oplus \mathbb{Z}^{k} =\mathbb{Z}^{k+1}$.
\end{proof}

\subsection{Twisted multipullback quantum complex projective spaces}

We begin by establishing notation. Fix a positive integer $N$, and let $\theta \in
M_{N+1}(\RR)$ be an antisymmetric real matrix. For $k,l \le N$, define $\Theta_{kl}
:=e^{2\pi i\theta_{kl}}$. For $0 \le k \le l \le N$, let $\Tt_{k,l}$ be the universal
$C^*$-algebra generated by the isometries $s_k,\ldots, s_l$ satisfying the usual
identities:
\begin{equation*}
s_is_j=\Theta_{ij}s_js_i,\ s_i^*s_j=\overline{\Theta}_{ij}s_js_i^*.
\end{equation*}
We will identify $\Tt_{k,l}$ with the corresponding subalgebra of $\Tt_{0,N}$. Let
$\Kk_{(k,l)}$ be the ideal of $\Tt_{k,l}$ generated by the product $\prod_{i=k}^l
(1-s_is_i^*)$. For each $k \le N$, the universal property of $\Tt_{0,N}$ shows that the
formula
$$
\alpha_k(s_i) := \Theta_{ik}s_i
$$
defines actions $\alpha_k$ of both $\NN$ and $\ZZ$ on $\Tt_{0,N}$, and hence on each
$\Tt_{l_1,l_2}$.

The idea of the computation is the same as in the untwisted case, with small changes due
to the fact that the isometries generating the noncommutative sphere do not commute. We
regard the twisted noncommutative sphere as the quotient of the twisted semigroup algebra
of $\NN^{N+1}$ by the ideal of compact operators: $C^*(\NN^{N+1},\Theta)/\Kk$. A
convenient presentation of $C^*(\NN,\Theta)$ that will be used below comes from the fact
that
$$
C^*(\NN^{N+1},\Theta)\cong (\ldots((\Tt\rtimes_{\alpha_1}\NN)\rtimes\NN)\ldots )\rtimes_{\alpha_{N}} \NN ,
$$
where the actions $\alpha_k$ are determined by the cocycle $\Theta$. While there exists a
considerable theory of semigroup $C^*$-algebras, we do not need to use it below. Instead,
we will reduce the computation to the one done in the untwisted case.

Let $\mu=(\mu_{k},\ldots,\mu_l )\in\NN^{l+1-k}$ be a multi-index, and let $\{e_\mu\}_\mu$
be the standard orthonormal basis of $l^2(\NN^{l+1-k})$. For $k \le i \le l$, let
$\delta_i := (0,\ldots,1,\ldots 0) \in \NN^{l+1-k}$ with $1$ in the slot labeled by $i$.
Define
$$
\pi_{(k,l)}(s_i)e_\mu :=\prod_{k\leq i<j\leq l}\Theta_{ij}^{\mu_j}e_{\mu +\delta_i}.
$$

\begin{lem}\label{lem:tech1}
Let $k\in\{1,\ldots,N\}$. In the decomposition
$$
l^2(\NN^{N+1})=l^2(\NN^{k})\otimes
l^2(\NN^{N+1-k}),
$$
 where the second factor corresponds to the last $N+1-k$ components in
$\NN^{N+1}$, the following equalities hold:
\begin{gather*}
\pi_{(0,N)}(\Tt_{0,k-1} \Kk_{(k,N)})=\pi_{(0,k-1)}(\Tt_{0,k-1})\otimes \Kk_{(k,N)},\\
\pi_{0,N} (\Kk_{(0,N)})=\pi_{(0,k-1)}(\Kk_{(0,k-1)})\otimes \Kk_{(k,N)}.
\end{gather*}
\end{lem}
\begin{proof}
By construction, for $i<k$,
\begin{gather*}
\pi_{(0,N)}(s_i)\in \pi_{(0,k-1)}(\Tt_{0,k-1}) \otimes_{\min} B(l^2(\NN^{N+1-k})),\\
\pi_{(0,N)}(\Kk_{(k,N)})\subset \pi_{(0,k-1)}(\Tt_{0,k-1})\otimes_{\min} \Kk (l^2(\NN^{N+1-k})).
\end{gather*}
Now the claim of the lemma follows.
\end{proof}

\begin{cor}
Let $k\in\{1,\ldots,N-1\}$. Put $C(S^{2k-1}_{H,\Theta_{0j}}):=\Tt_{0,j}/\Kk_{(0,j)}$.
There exists a U(1)-equivariant short exact sequence of $C^*$-algebras:
\begin{multline*}
\xymatrix{
    0\ar[r]&
        C(S^{2k-3}_{H,\Theta_{0\,(k-1)}})\otimes \Kk_{(k,N)}\ar[r]&
        C(S^{2k-1}_{H,\Theta_{0 k}})\otimes {\Kk_{(k+1,N)}}
        }\\
\xymatrix{
    \ar[r]&
    \left( \Tt_{0,k-1}\rtimes_{\alpha_{k}}\ZZ \right) \otimes {\Kk_{(k+1,N)}}\ar[r]&
    0.
}
\end{multline*}
The action of U(1) is the one induced naturally from its diagonal action on $\Tt_{0,N}$.
\end{cor}
\begin{proof}
Lemma~\ref{lem:tech1} reduces the claim to the identity
\begin{multline*}
\pi_{(0,k)}(\Tt_{0,k})/\pi_{(0,k)}(\Tt_{0,k-1}\Kk_{0,k})\\
=\pi_{(0,k)}(\Tt_{0,k-1}\Tt_{k,k})/\pi_{(0,k)}(\Tt_{0,k-1}\Kk_{(k,k)})\cong \Tt_{0,k-1}\rtimes_{\alpha_k} \ZZ,
\end{multline*}
which immediately follows  from the construction of~$\pi_{(0,k)}$.
\end{proof}

\begin{proof}[Proof of Theorem~\ref{thm:main}(\ref{it:main-PS-Kth})]
For $0\leq k \le N$, let $T_k := (C(S^{2k+1}_{\Theta_{0 k}})\otimes {\Kk_{(k+1,N)}})^{U(1)}$.
Since the crossed product $\Tt_{0,k-1}\rtimes_{\alpha_{k}}\ZZ$ contains the regular
representation of $\ZZ$, and hence a copy of the regular representation of U(1) on
$C^*(\ZZ)=C(S^1)$, we get, as in the untwisted case,
$$
\big(\left( \Tt_{0,k-1}\rtimes_{\alpha_{k}}\ZZ \right) \otimes {\Kk_{(k+1,N)}}\big)^{U(1)}\cong  \Tt_{0,k-1} \otimes {\Kk_{(k+1,N)}}.
$$
Finally, as $T_N\cong C(\mathbb{P}^N_\theta(\Tt))$ by Theorem~\ref{prp:pullback
isomorphism}, the rest of the argument is the same as in the untwisted case, with $T_k$
in place of~$S_k$.
\end{proof}

\section{Noncommutative line bundles over multipullback quantum complex projective spaces}\label{sec:line bundles}

\subsection{Equivariant homomorphisms and spectral subspaces}

Take a $U(1)$-equi\-var\-iant *-homomorphism $f\colon A\to A'$ of unital
$U(1)$-$C^*$-algebras, and suppose that the $U(1)$-action on $A$ is free. Then there
exists a strong connection $\ell$ on~$A$. It is straightforward to check that
$\ell':=(f\otimes f)\circ\ell$ is a strong connection on~$A'$, so that the $U(1)$-action
on $A'$ is also free. The $U(1)$-equivariance of $f$ guarantees that its restriction to
the fixed-point subalgebra $B:=A^{U(1)}$ corestricts to the fixed-point subalgebra
$B':=(A')^{U(1)}$. This $f$ turns $B'$ into a $B'$--$B$ bimodule given by the usual
multiplication on the left and the formula $b'\cdot b:=b'f(b)$ on the right.

Since $f : A \to f(A)$ is a linear surjection over a field, it splits. So there exists a
linear map $g\colon f(A)\to A$ such that $f\circ g=\id_{f(A)}$. We have $A=g(f(A))\oplus
\ker f$. Let $\{a'_j\}_j$ be an extension of a basis $\{e'_i\}_i$ of $f(A)$ to a basis
of~$A'$. Also, let $\{e_k\}_k$ be a basis of~$\ker f$. Then $\{a_l\}_l:=\{g(e'_i)\}_i\cup
\{e_k\}_k$ is a basis of~$A$, and $f(a_l)=a'_l$ or $f(a_l)=0$. For any $n\in\mathbb{Z}$,
we can write $\ell(u^n)=\sum_{l\in L}a_l\otimes r_l(u^n)$ and $\ell'(u^n)=\sum_{l\in
L'}a'_l\otimes f(r_l(u^n))$. Here $L'$ and $L$ are respectively $m'$ and $m$ element
sets, with $m'\leq m$, and $f(a_l)=a'_l$ for $l\leq m'$ and $f(a_l)=0$ for $l>m'$.

It follows from the Chern-Galois theory of~\cite{bh04} that the existence of a strong
connection guarantees that spectral subspaces are finitely generated projective as left
modules over fixed-point $C^*$-algebras. Given a strong connection $\ell$ and a spectral
subspace $A_n$, we have an explicit formula given in \cite[Theorem~3.1]{bh04} for an
idempotent $E^n$ representing the spectral subspace: $E_{kl}^n:= r_k(u^n)a_l$. Hence
$f(E_{kl}^n)=f(r_k(u^n))a'_l$ for $l\leq m'$ and $f(E_{kl}^n)=0$ for $l>m'$ are the
matrix coefficients of an idempotent representing $B'\otimes_B A_n$. Using the strong
connection $\ell'$ and the linear basis $\{a'_l\}_l$, we conclude that the matrix
coefficients of an idempotent representing $A'_n$ are also $f(r_k(u^n))a'_l$, but with
indices $k,l\in L'$.

To continue this reasoning and to take care of the range of indices,
 it is convenient to adopt the block-matrix notation. Let $\beta_n:=(r_1(u^n),\ldots,r_m(u^n))$
and $\gamma:=(a_1,\ldots,a_m)$. Much in the same way, let $\beta'_n:=(f(r_1(u^n)),\ldots,f(r_{m'}(u^n)))$
and $\gamma':=(a'_1,\ldots,a'_{m'})$.
Then $E^n= {\beta_n}^T\gamma\in M_m(B)$ is an idempotent matrix representing~$A_n$, and
$(E')^n= {\beta'_n}^T\gamma'\in M_{m'}(B')$ is an idempotent matrix representing~$A'_n$.
Finally, put
\begin{gather*}
\beta''_n:=(f(r_1(u^n)),\ldots,f(r_{m}(u^n))=:(\beta'_n,\rho'_n)\qquad\text{and}
\\
\nonumber\gamma'':=(\gamma',0,\ldots,0)\qquad \text{(with $m-m'$ zeros at the end)}.
\end{gather*}

\noindent
 Then $(E'')^n= {\beta''_n}^T\gamma''\in M_m(B')$ is an idempotent matrix
representing~$B'\otimes_BA_n$.

The crux of our argument is that $(E')^n$ and $(E'')^n$ represent \emph{isomorphic} left
$B'$-modules. After extending $(E')^n$ by zeros to size $m$, we obtain a matrix
conjugate\footnote{ We are grateful to Tomasz Maszczyk for pointing this out to us.}
to~$(E'') ^n$:
\begin{equation*}
\left( \begin{array}{cc}
1& 0  \\
-{\rho'_n}^T\gamma' & 1 \end{array} \right)
\left( \begin{array}{c}
{\beta'_n}^T\\
{\rho'_n}^T \end{array} \right)
\overset{\left( \begin{array}{cc}
\gamma' & 0  \end{array} \right)}
{\phantom{\text{\tiny x}}}
\left( \begin{array}{cc}
1& 0  \\
{\rho'_n}^T\gamma' & 1 \end{array} \right)=
\left( \begin{array}{c}
{\beta'_n}^T\\
0\end{array} \right)
\overset{\left( \begin{array}{cc}
\gamma' & 0  \end{array} \right)}
{\phantom{\text{\tiny x}}}.
\end{equation*}

\noindent
Here we used the fact that $\gamma'{\beta'_n}^T=1$, which is
condition~\eqref{eq:m circ l} for the strong connection~$\ell'$.
Following the reasoning of the previous paragraph we have arrived at:

\begin{thm}\label{thm:epullback}
Let \mbox{$f\colon A\to A'$} be a $U(1)$-equivariant *-homomorphism of unital
$U(1)$-$C^*$-algebras, and let $B$ and $B'$ be the respective fixed-point
$C^*$-subalgebras. Assume that the $U(1)$-action on $A$ is free. For each $n \in
\mathbb{Z}$, let $A_n$ and $A'_n$ denote the $n$\textsuperscript{th} spectral subspaces
of $A$ and $A'$ respectively. Then, for any $n\in\mathbb{Z}$, there is an isomorphism of
finitely generated left $B'$-modules:
\begin{equation*}
B'\underset{B}{\otimes}A_n\cong A'_n.
\end{equation*}
In particular, the induced map $(f|_B)_*:K_0(B)\to K_0(B')$ satisfies
\begin{equation*}
(f|_B)_*\big([A_n]\big)=[A'_n] \quad\text{ for every } n\in\mathbb{Z}.
\end{equation*}
\end{thm}

\subsection{Pairwise non-isomorphism}

The goal of this section is to prove Theorem~\ref{thm:main}(\ref{it:main-stabnontriv}), i.e.~to show that
the line bundles over the multipullback quantum complex projective space $\mathbb{P}^N(\Tt)$ associated to
the Heegaard odd quantum sphere $S^{2N+1}_H$ are classified by their defining winding number.
 We will do it reducing the problem to the special case $N=1$, which was already solved elsewhere.
Here the main problem is that we do not have any $U(1)$-equivariant maps from $C(S^{2N+1}_H)$ to~$C(S^{3}_H)$.
We overcome this difficulty by finding a wrong-way equivariant map that restricted to fixed-point subalgebras induces
 an isomorphism on the K-groups.

To begin with, we need to unravel the pullback structure of~$C(S^{2N+1}_H)$:
\begin{lem}
\label{SphereIterate}
For any $N\in\mathbb{N}$, $N>0$,  the $U(1)$-$C^*$-algebra $C(S^{2N+1}_H)$ can be presented as the following
equivariant pullback:
\begin{equation*}
\xymatrix{
&C(S^{2N+1}_H)\ar[dl]_{\text{\rm pr}^N_1}\ar[dr]^{\text{\rm pr}^N_2}&\\
C(S^{2N-1}_H)\otimes\Tt\ar[dr]_{\id\otimes\sigma} && \Tt^{\otimes N}\otimes C(S^1)\ar[dl]^{\bm{\sigma}\otimes\id}\\
&C(S^{2N-1}_H)\otimes C(S^1)\,. &
}
\end{equation*}
Here $\bm{\sigma}\colon \Tt^{\otimes N}\ni w\mapsto (\sigma_0(w),\dots,\sigma_{N-1}(w) )\in C(S^{2N-1}_H)$,
and $\sigma_i$ is defined by~\eqref{sigmai}, which is the $\theta=0$ case of~\eqref{sigmaitheta}. The defining *-homomorphisms
are equivariant with respect to the diagonal action.
\end{lem}
\begin{proof}
We adopt the definitions from \eqref{sigmai} and \eqref{piij}, but now we have
to play with different $N$ at the same time, whence the need for additional labeling:
\begin{align*}
B^N_i &:=\Tt^{\otimes i}\otimes C(S^1)\otimes\Tt^{\otimes (N-i)}\\
\pi^{i,N}_j &:= \id^j\otimes\sigma\otimes\id^{N-j}.
\end{align*}
Then, the definition of $C(S^{2N+1}_H)$ becomes:
\begin{equation*}\textstyle
C(S^{2N+1}_{H}) := \Big\{(b_0, \dots, b_N) \in \bigoplus^N_{i=0} B_i^N \;\big|\;\forall\;0 \le i < j \le N:\; \pi^{i,N}_j(b_i)
= \pi^{j,N}_i(b_j)\Big\}.
\end{equation*}
Denoting
$
\forall\; 0\leq i<j\leq M\leq N:\;\pi^{i, N}_j(b_i)=\pi^{j, N}_i(b_j)
$ by $P_M^N((b_i)_i)$,
we can rewrite this formula as
\begin{multline}
\textstyle
C(S^{2N+1}_H) = \Big\{((b_i)_i, b_N)\in(\bigoplus_{i=0}^{N-1}B_i^N)
\oplus(\Tt^{\otimes N}\otimes C(S^1))\;\big|
\\
\label{HeegN}
 P^N_{N-1}((b_i)_i)\textstyle{}\wedge\bigl(\forall\; 0\leq i\leq N-1:\;\pi^{i,N}_N(b_i) = (\sigma_i\otimes\id)(b_N)\bigr)\Big\}.
\end{multline}

Next, using the exactness of the tensor product $\_\otimes \Tt$ (which follows from nuclearity of
$\Tt$), we can write
\begin{align}
C(S^{2N-1}_H)\otimes\Tt
=&\textstyle\Big\{(\tilde b_i)_i\in\bigoplus_{i=0}^{N-1}B^{N-1}_i\;|\;P^{N-1}_{N-1}((\tilde b_i)_i) \Big\}\otimes\Tt
\nonumber\\
=&\textstyle\Big\{(b_i)_i\in\bigoplus_{i=0}^{N-1}B^{N}_i\;|\;P^{N}_{N-1}((b_i)_i) \Big\}.
\label{HeegN2}\end{align}
Combing \eqref{HeegN} with \eqref{HeegN2}, we arrive at:
\begin{multline*}
C(S^{2N+1}_H) =
\big\{((b_i)_i, b_N)\in\bigl(C(S^{2N-1}_H)\otimes\Tt\bigr)
\oplus\bigl(\Tt^{\otimes N}\otimes C(S^1)\bigr)\;|\;\\
\forall\;0\leq i\leq N-1:\;\pi^{i,N}_N(b_i) = (\sigma_i\otimes\id)(b_N)\big\}.
\end{multline*}
Finally, we obtain
\begin{align*}
C(&S^{2N+1}_H) \\
&= \big\{(x, y)\in\bigl(C(S^{2N-1}_H)\otimes\Tt\bigr)
\oplus\bigl(\Tt^{\otimes N}\otimes C(S^1)\bigr)\;|\;
(\id\otimes\sigma)(x)=(\bm{\sigma}\otimes\id)(y)\big\},
\end{align*}
which proves the lemma.
\end{proof}

The next step is to establish a wrong-way map with the right-way inverse in $K$-theory:
\begin{lem}
\label{IsomLemKGroups}
Consider $C(S^3_H)\otimes\Tt^{\otimes (N-1)}$ with the diagonal $U(1)$-action.
Then
$$
\eta\colon C(S^3_H)\ni x\longmapsto x\otimes 1\in C(S^3_H)\otimes\Tt^{\otimes (N-1)}
$$
is a $U(1)$-equivariant *-homomorphism whose restriction-corestriction $\bar\eta$
to the $U(1)$-invariant subalgebras
induces an isomorphism of $K$-groups:
$$
{\bar\eta}_{*}:K_{*}\big(C(\mathbb{P}^1(\Tt))\big)\longrightarrow
K_{*}\Big(\big(C(S^3_H)\otimes\Tt^{\otimes (N-1)}\big)^{U(1)}\Big).
$$
\end{lem}
\begin{proof}
The pullback presentation of $C(S^3_H)$ together with the exactness of tensoring with
$\Tt^{\otimes (N-1)}$ yields two $U(1)$-equivariant pullback diagrams. We combine them in
the following commutative diagram of $U(1)$-equivariant *-homomorphisms (all considered
with the diagonal $U(1)$-action):
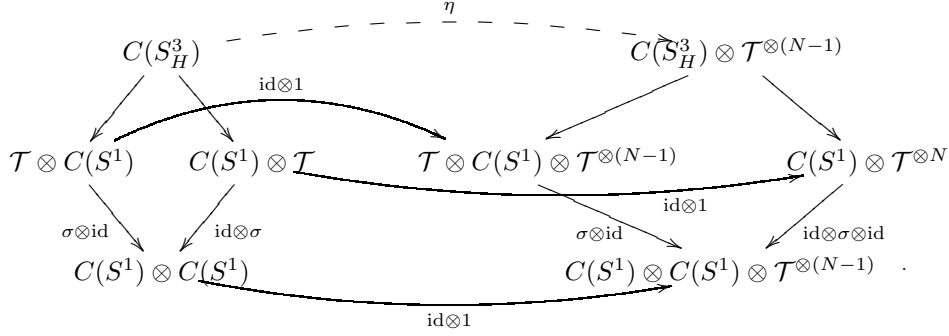
\begin{equation}\label{diagram1}
\xymatrix@C=0pt{
&C(S^3_H)\ar[ld]\ar[rd]\ar@/^1pc/@{-->}[rrrr]^{\eta}& && &
\makebox[10ex]{$C(S^3_H)\otimes\Tt^{\otimes(N-1)}$}\ar[ld]\ar[rd]&\\
\makebox[6ex]{$\Tt\otimes C(S^1)$}\ar[dr]_{\sigma\otimes \id}
\ar@/^2pc/[rrrr]^{\id\otimes 1}
&&  \makebox[6ex]{$C(S^1)\otimes\Tt$}
\ar@/_1pc/[rrrr]_(0.7){\id\otimes 1}
\ar[dl]^{\id\otimes\sigma} &&
\makebox[30ex]{$\quad\Tt\otimes C(S^1)\otimes\Tt^{\otimes(N-1)\!\!\!\!\!\!\!\!\!\!\!\!\!\!\!\!\!\!\!\!\!\!\!\!}$}\ar[dr]_{\sigma\otimes\id} &&
\makebox[10ex]{$C(S^1)\otimes\Tt^{\otimes N}$}\ar[dl]^{\id\otimes\sigma\otimes\id}\\
&\makebox[5ex]{$C(S^1)\otimes C(S^1)$}
\ar@/_1pc/[rrrr]_{\id\otimes 1}
& && & \makebox[10ex]{$C(S^1)\otimes C(S^1)\otimes\Tt^{\otimes(N-1)\quad.}$}&
}
\end{equation}

Using the gauge isomorphisms~\eqref{kappa}  together with some permutations of
tensor factors, we  transform the diagonal action (on the pullback components)
to the action on the  rightmost factor thus obtaining the following diagram:
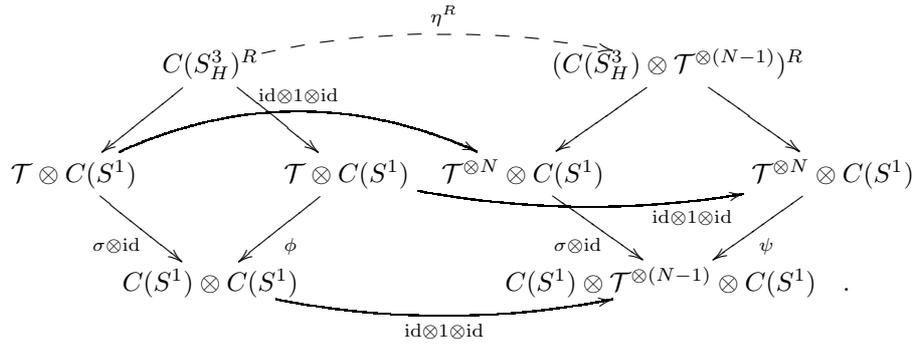
\begin{equation}\label{diagram2}
\xymatrix@C=0pt{
&C(S^3_H)^R\ar[ld]\ar[rd]\ar@/^1pc/@{-->}[rrrr]^{\eta^R}& && &
\makebox[10ex]{$(C(S^3_H)\otimes\Tt^{\otimes (N-1)})^R$}\ar[ld]\ar[rd]&\\
\Tt\otimes C(S^1)\ar[dr]_{\sigma\otimes \id}
\ar@/^2pc/[rrrr]^{\id\otimes 1\otimes\id}
&&  \Tt\otimes C(S^1)
\ar@/_1pc/[rrrr]_(0.7){\id\otimes 1\otimes\id}
\ar[dl]^{\phi} &&
\Tt^{\otimes N}\otimes C(S^1)\ar[dr]_{\sigma\otimes\id} &&
\Tt^{\otimes N}\otimes C(S^1)\ar[dl]^{\psi}\\
& \makebox[10ex]{$C(S^1)\otimes C(S^1)$}
\ar@/_1pc/[rrrr]_{\id\otimes 1\otimes\id}
& && & \makebox[10ex]{$C(S^1)\otimes\Tt^{\otimes (N-1)}\otimes C(S^1)\quad.$}&
}
\end{equation}
Here the top line is $U(1)$-equivariantly isomorphic to the top line of the previous diagram, and
$\phi$ and $\psi$ are given by
\begin{align*}
\phi:\Tt\otimes C(S^1)&\longrightarrow C(S^1)\otimes C(S^1), \nonumber\\
\phi:t\otimes u&\longmapsto u\sw{1}S(\sigma(t))\otimes u\sw{2}\,,\nonumber\\
\psi:\Tt\otimes\Tt^{\otimes (N-1)}\otimes C(S^1)
&\longrightarrow C(S^1)\otimes \Tt^{\otimes (N-1)}\otimes C(S^1),\nonumber\\
\psi:t\otimes \bar{r}\otimes u&\longmapsto
S(\sigma(t)\bar{r}\sw{1})u\sw{1}\otimes \bar{r}\sw{0}\otimes u\sw{2}\,.
\end{align*}

Finally, to pass to the restriction-corestriction of Diagram~\ref{diagram1}
to the $U(1)$-invariant subalgebras, it suffices to note that it is isomorphic the restriction-corestriction of Diagram~\ref{diagram2},
and that the latter is obtained by removing the rightmost factors from the pullback components:
\begin{equation*}
\xymatrix@C=15pt{
&C(\mathbb{P}^1(\Tt))\ar[ld]\ar[rd]\ar@/^1pc/@{-->}[rrrr]^{\bar{\eta}}& && &
\makebox[10ex]{$\Big(\big(C(S^3_H)\otimes\Tt^{\otimes (N-1)}\big)^R\Big)^{U(1)}\qquad$}\ar[ld]\ar[rd]&\\
\Tt\ar[dr]_{\sigma}
\ar@/^2pc/[rrrr]^{\id\otimes 1}
&&  \Tt
\ar@/_1pc/[rrrr]_(0.7){\id\otimes 1}
\ar[dl]^{S\circ\sigma} &&
\Tt^{\otimes N}\ar[dr]_{\sigma\otimes\id} &&
\Tt^{\otimes N}\ar[dl]^{\tilde{\psi}}\\
&C(S^1)
\ar@/_1pc/[rrrr]_{\id\otimes 1}
& && & C(S^1)\otimes\Tt^{\otimes (N-1)}\quad.&
}
\end{equation*}
Here
\begin{equation*}
\tilde{\psi}:\Tt^{\otimes N}\ni t\otimes \bar{r}\longmapsto
S(\sigma(t)\bar{r}\sw{1})\otimes \bar{r}\sw{0}\in C(S^1)\otimes \Tt^{\otimes (N-1)}.
\end{equation*}
Due to  the naturality of the K\"unneth formula, all
three maps $\id\otimes 1_{\Tt^{\otimes (N-1)}}$ between the pullback components
induce isomorphisms on $K$-groups. Hence, it
follows from \cite[Theorem~3.1]{FHMZ} that also $\bar{\eta}$ induces an
isomorphism on $K$-groups.
\end{proof}

\begin{proof}[Proof of Theorem~\ref{thm:main}(\ref{it:main-stabnontriv})]
Lemma~\ref{SphereIterate} implies that
\[
f:=(\text{pr}_1^2\otimes\id_{\Tt^{\otimes (N-1)}})\circ
(\text{pr}_1^3\otimes\id_{\Tt^{\otimes (N-2)}})\circ\cdots\circ
\text{pr}_1^N
\]
is a surjective $U(1)$-equivariant *-homomorphism
\begin{equation*}
f : C(S^{2N+1}_H)\longrightarrow C(S^3_H)\otimes\Tt^{\otimes (N-1)}\!.
\end{equation*}
Furthermore, by Lemma~\ref{IsomLemKGroups} we have a $U(1)$-equivariant
*-homomorphism
$$
\eta:C(S^3_H)\longrightarrow C(S^3_H)\otimes \Tt^{\otimes (N-1)},
$$
which induces an isomorphism on $K$-groups.

Next, the freeness of the diagonal $U(1)$-action on $C(S^{2N+1}_H)$,
which follows from Section~\ref{3.2} for $\theta=0$, allows us to apply the final statement of
Theorem~\ref{thm:epullback} to  infer that the equality of $K_0$-classes
$[C(S^{2N+1}_H)_m]=[C(S^{2N+1}_H)_n]$ implies
the equality of $K_0$-classes
\begin{multline}\label{f}
\bigl[(C(S^3_H)\otimes \Tt^{\otimes (N-1)})_m\bigr]=\bar{f}_{*}([(S^{2N+1}_H)_m])\\
=\bar{f}_{*}([(S^{2N+1}_H)_n])
=\bigl[(C(S^3_H)\otimes \Tt^{\otimes (N-1)})_n\bigr].
\end{multline}
Here by $\bar{f}$ we denoted the restriction-corestriction of $f$
to $U(1)$-invariant subalgebras. Much in the same way, identifying the isomorphic $C^*$-algebras
$$
\Big(\big(C(S^3_H)\otimes\Tt^{\otimes (N-1)}\big)^R\Big)^{U(1)}\cong
\big(C(S^3_H)\otimes\Tt^{\otimes (N-1)}\big)^{U(1)},
$$
we conclude that
\begin{align}
\bigl[(C(S^3_H)\otimes \Tt^{\otimes (N-1)})_m\bigr]&=\bar\eta_{*}\bigl([C(S^3_H)_m]\bigr),\\
\bigl[(C(S^3_H)\otimes \Tt^{\otimes (N-1)})_n\bigr]&=\bar\eta_{*}\bigl([C(S^3_H)_n]\bigr).\label{eta}
\end{align}
Now, it follows from \eqref{f}--\eqref{eta} and Lemma~\ref{IsomLemKGroups} that
$[C(S^3_H)_m]= [C(S^3_H)_n]$. Finally, by an index-pairing calculation\cite[Theorem~3.3]{hms03}, we obtain $m=n$.
\end{proof}


\begin{center}{\sc Acknowledgements}\end{center}
\noindent
This research was supported by NCN grant 2012/06/M/ST1/00169 and the Australian Research
Council. Ryszard Nest was supported by the Danish National Research Foundation (DNRF) through
the Centre for Symmetry and Deformation at the University of Copenhagen (CPH-SYM-DNRF92).
Part of this work was completed when two of the authors (P.~M.~Hajac and
A.~Sims) were at the University of M\"unster, supported by the Focus Programme on
$C^*$-algebras in July 2015. We thank the organizers for the opportunity. It is a
pleasure to thank Carla Farsi, Elizabeth Gillaspy  and Tomasz Maszczyk for  discussions.

\end{document}